\documentclass[11pt]{amsart}

\usepackage{amsmath}
\usepackage{amssymb}
\usepackage{amsthm}
\usepackage[mathscr]{eucal}
\usepackage{mathrsfs}
\usepackage{psfrag}
\usepackage[margin=1in]{geometry}
\usepackage{color}
\usepackage{eucal}
\usepackage[dvips]{graphicx}
\usepackage{amsfonts}%
\usepackage[nobysame]{amsrefs}
\usepackage{hyperref}

\theoremstyle{plain}
\newtheorem{theorem}{Theorem}
\newtheorem{lemma}{Lemma}

\newtheorem{condition}[theorem]{Condition}

\theoremstyle{definition}
\newtheorem{definition}[theorem]{Definition}

\theoremstyle{remark}

\DeclareMathOperator{\divop}{div}

\newcommand{\ud}{\,\mathrm{d}}

\newcommand{\RR}{\mathbb{R}}

\newcommand{\PP}{\mathbb{P}}
\newcommand{\EE}{\mathbb{E}}
\DeclareMathOperator{\std}{Std}

\newcommand{\wt}[1]{\widetilde{#1}}

\IfFileExists{mathabx.sty}%
  {\DeclareFontFamily{U}{mathx}{\hyphenchar\font45}%
   \DeclareFontShape{U}{mathx}{m}{n}{<->mathx10}{}%
   \DeclareSymbolFont{mathx}{U}{mathx}{m}{n}%
   \DeclareFontSubstitution{U}{mathx}{m}{n}%
   \DeclareMathAccent{\widebar}{0}{mathx}{"73}%
}{%
  \PackageWarning{mathabx}{%
    Package mathabx not available, therefore\MessageBreak substituting
    widebar with overline\MessageBreak }%
  \newcommand{\widebar}[1]{\overline{#1}}%
}
\newcommand{\wb}[1]{\widebar{#1}}

\newcommand{\mc}[1]{\mathcal{#1}}

\newcommand{\veps}{\varepsilon}

\newcommand{\average}[1]{\left\langle#1\right\rangle}

\title[Analysis of multiscale integrators for multiple attractors and irreversible Langevin samplers]{Analysis of multiscale integrators for multiple attractors and
 irreversible Langevin samplers}

\author{Jianfeng Lu}
\address{Department of Mathematics, Department of
  Physics, and Department of Chemistry, Duke University, Box 90320,
  Durham NC 27708, USA}
\email{jianfeng@math.duke.edu}

\author{Konstantinos Spiliopoulos}
\address{Department of Mathematics and Statistics, Boston University, Boston MA 02446,  USA}
\email{kspiliop@math.bu.edu}

\thanks{J.L. was partially supported by the National Science
  Foundation under the grant DMS-1415939. K.S. was partially supported
  by the National Science Foundation (NSF)  CAREER award DMS 1550918. }

\date{\today}
\begin{document}

\begin{abstract}
We study multiscale integrator numerical schemes for a class of stiff stochastic differential equations (SDEs). We consider multiscale SDEs with potentially multiple attractors that behave as diffusions on graphs  as the stiffness parameter goes to its limit. Classical numerical discretization schemes, such as the Euler-Maruyama scheme, become unstable as the stiffness parameter converges to its limit and appropriate multiscale integrators can correct for this. We rigorously establish the convergence of the numerical method to the related diffusion on graph, identifying the appropriate choice of discretization parameters. Theoretical results are supplemented by numerical studies on the problem of the  recently developing area of introducing irreversibility in Langevin samplers in order to accelerate convergence to equilibrium.
\end{abstract}

\maketitle

\section{Introduction}

The main focus of this work is numerical integrators for stochastic
differential equations (SDEs) with multiscale coefficients, with the focus on irreversible first order (overdamped) Langevin dynamics with additive noise.
%
The main motivation
of this work is to design numerical integrators for SDEs arising from
recent works on irreversible Langevin samplers 
\cites{HwangMaSheu2005,ReyBelletSpiliopoulos2014,ReyBelletSpiliopoulos2015}.  In
those works, an irreversible drift term is added to the overdamped
Langevin dynamics (details will be specified in
section~\ref{sec:hmm}), and it is proved that under the proper assumptions the sampling efficiency
increases as the magnitude of the irreversible drift goes to infinity,
which is validated by numerical studies, see
\cites{ReyBelletSpiliopoulos2014,ReyBelletSpiliopoulos2015,DuncanPavliotisLelievre2016, DuncanPavliotisZygalakis2017}. However,
at the same time, when a strong irreversible drift is added to the
original overdamped Langevin equation, the stiffness of the system is
inevitably increased, and thus prevents the application of standard
numerical integrators to the resulting systems. The goal of this work
is to study multiscale integrators that allow to enlarge the magnitude
of the irreversible drift without having to sacrifice the stability of
the numerical algorithm.

For SDEs with multiscale coefficients, it is well understood that
we shall take into account the multiscale structure in order to design better
integrators (see e.g., the books \cite{PavliotisStuart:08, E:11}). The key idea is to use the averaged limit of the
SDE when the scale is well separated. Hence,  it is not necessary
to accurately resolve the scales of the original system, but we can rather work with the
averaged limit. This has been the underlying principle of the
heterogeneous multiscale methods (HMM) \cite{EEngquist:03, HMMReview1, HMMReview2}, in particular see \cite{EVE:03, EVE:07, ELiuVE:05} for its
applications to stochastic differential equations. Other numerical approaches for stiff SDEs were
also developed in \cite{AbdulleCirilli:08, AbdulleLi:08,
  TaoOwhadiMarsden2010, BurrageTian:01, AbdulleZygalakis:12,
  GivonKevrekidisKupferman:06}.

Our numerical scheme follows the ideas of the FLAVORS method developed
in \cite{TaoOwhadiMarsden2010}, which on the algorithmic level is very
similar to the seamless version of HMM method developed in
\cite{ELu:07, ERenVE:09}. The basic idea is to use a split-step
integrator which combines a short time integration of the whole SDE
and a longer time integration of the SDE without the stiff terms. The
numerical analysis of such schemes \cite{TaoOwhadiMarsden2010} shows
that in the case that the variables of the SDEs can be one-to-one
mapped to a set of ``fast'' and ``slow'' variables, the numerical
scheme converges to the averaged limit which consists of the dynamics
of the slow component. We emphasize that the algorithm does not
require explicit knowledge of the mapping that transforms the system
into fast and slow variables, while it does require the forcing terms
of the SDE can be separated into stiff and non-stiff terms.

The main contribution of this work is to extend the analysis of
\cite{TaoOwhadiMarsden2010} to situations that a one-to-one mapping of
the original degree of freedom into fast and slow variables is not
possible. In particular, for the irreversible Langevin sampler the function $U$ that maps the configurational
space to the energy is clearly not one-to-one. In fact, it is well
known that in the limit the SDE converges to a diffusion on an associated graph
\cite{FreidlinWentzell88, FreidlinWentzell1993}, for which besides the energy, one has to add the index
variable to represent the state space. Our main result proves that the
multiscale integrator converges to a  diffusion on graph as the
scale separation parameter tends to infinity and the discretization parameters are appropriately chosen. In the one well case, our proof follows ideas
of \cite{TaoOwhadiMarsden2010} appropriately adjusting for the different limiting behavior that we have here. Then, the results are being extended to the multiple well case by using techniques similar to those of the classical
averaging techniques of \cite{BrinFreidlin2000,FreidlinWeber2004,
  FreidlinWentzell1993}. However, since we work in the discrete time framework and not in the continuous time framework,  we need to obtain bounds with explicit dependence on the discretization parameters.

In this paper, we mainly study convergence  to the invariant measure of the limiting dynamics.  The mathematical analysis suggests how to choose the parameters of the problem (micro step and macro step) with respect to a given  value of the stiffness parameter in order for the HMM integrator to sample from the correct measure. In addition, being able to numerically approach the limit of the stiffness parameter allows us to approximate via simulation
the limiting transition probabilities between the different attractors of the system. 

In regards to future research, it would be of great interest to address the challenges that come up in convergence in finite time points, as it done in \cite{HutzenthalerJentzenKloeded2011} for Euler's method. It would also be of great interest to obtain nonasymptotic bounds in the spirit, for example, of \cite{DurmusMoulines2017}.

This paper is organized as follows. We will introduce the SDEs from
the irreversible Langevin sampler and the HMM multiscale integrator in
Section~\ref{sec:hmm}. Some numerical results are presented in
Section~\ref{sec:numerics} to validate the method. The averaging
results of the SDEs, in particular, convergence to the diffusion on
graphs are recalled in Section~\ref{sec:averaging}. The main results and the proofs are given in Section~\ref{sec:analysis}.

\section{HMM integrator for irreversible Langevin sampling
  scheme}\label{sec:hmm}

Consider the overdamped Langevin equation
\begin{equation}\label{eq:overdamp}
  \ud Z_t = -\nabla U(Z_t) \ud t + \sqrt{2 \beta} \ud W_t, Z_{0}=z_{0},
\end{equation}
where $U: E \to \RR$ is a given potential, $\beta = (k_B T)^{-1}$ is
the temperature, and $W_t$ is the standard multi-dimensional
Wiener process. Here $E\subseteq \mathbb{R}^{d}$ denotes the state space.
See Section~\ref{sec:averaging} for conditions on $U$. The overdamped
Langevin dynamics \eqref{eq:overdamp} is often used to sample the
Boltzmann-Gibbs measure, see \cite{RobertsTweedie:96}, with density given by
\begin{equation}
  \varrho(z) \propto e^{- U(z) / \beta},\nonumber
\end{equation}
which is the invariant measure of \eqref{eq:overdamp} under mild
conditions. Note that the infinitesimal generator of \eqref{eq:overdamp}
is symmetric with respect to the invariant measure, and thus the
dynamics \eqref{eq:overdamp} is reversible in time, \textit{i.e.}, it
satisfies detailed balance.

In \cite{HwangMaSheu2005,ReyBelletSpiliopoulos2014,
  ReyBelletSpiliopoulos2015}, it was proposed to add to the
overdamped Langevin dynamics an irreversible forcing to accelerate the
sampling, the resulting dynamics reads
\begin{equation}\label{eq:irrevSDE}
  \ud Z^{\veps}_t = \bigl[ - \nabla U(Z^{\veps}_t) + \frac{1}{\veps} C( Z^{\veps}_t) \bigr] \ud t + \sqrt{2 \beta} \ud W_t,  Z_{0}=z_{0},
\end{equation}
where the vector field $C:\mathbb{R}^{d}\mapsto\mathbb{R}^{d}$.  The invariant measure is maintained if the vector fields $C$  satisfies $\divop( Ce^{-U/\beta})=0$, or equivalently
\[
\divop C = \beta^{-1} C \cdot \nabla U \,,
\]
where $ C \cdot \nabla U$ denotes the classical inner product between $C$ and $\nabla U$. A convenient choice, which we assume henceforth, is to pick $C$ such that
\[
\divop C =0 \,, \quad {\rm and} \quad  C \cdot \nabla U =0 \,.
\]

This is not the most general choice for $C$, but it has the advantage
that allows to choose $C$ independently of $\beta$. One such choice of
$C(z)$ is $C(z)=J\nabla U(z)$, where $J$ is any antisymmetric matrix.
These conditions mean that the flow generated by $C$ preserves
Lebesgue measure since it is divergence-free, at the same time, since
$U$ is a constant of the motion, the micro-canonical measure on the
surfaces $\{U = z\}$ are preserved as well.  Let us
  remark that in physics terminology, a Langevin equation is a
  second-order dynamics which also include momentum variables in
  addition to the ``position variable'' $Z_t$ as in the overdamped
  equation \eqref{eq:overdamp}. The physical Langevin equation is in
  fact irreversible due to the momentum degree of freedom, while here
  we have adopted the conventional name of irreversible Langevin
  sampler for the first-order dynamics \eqref{eq:irrevSDE} with
  additional irreversible drift on the overdamped Langevin equations.

The amplitude of the irreversible drift in \eqref{eq:irrevSDE} is
chosen to be $\frac{1}{\veps}$. We will consider the regime that
$\veps \ll 1$. Using the large deviation action functional of the
empirical measure, it is shown in
\cite{ReyBelletSpiliopoulos2014,ReyBelletSpiliopoulos2015} that the
dynamics \eqref{eq:irrevSDE} converges faster to the invariant measure
for a larger irreversible drift, i.e., as $\veps$ becomes
smaller. From another point of view, as will be recalled in
section~\ref{sec:averaging}, in the limit $\veps \to 0$, the slow component associated to the solution of the SDE
\eqref{eq:irrevSDE} converges to the averaging limit which is a
diffusion on an associated graph, and hence the entropy associated
with the iso-surfaces of $U$ is completely removed and only the
energetic barrier is left in the limit.\footnote{While it is possible
  to combine the irreversible sampling with other techniques to
  overcome the energetic barrier, we will not go further in this
  direction as it is not the focus of the current work.} In
\cite{ReyBelletSpiliopoulos2015} it is also established that the
asymptotic (as $t\rightarrow\infty$) variance of the estimator is
decreasing in $\veps$ and in the limit as $\veps\downarrow 0$, it
converges to the asymptotic (as $t\rightarrow\infty$) variance of the
corresponding sampling problem on the graph where the limiting
diffusion lives.

Increasing the irreversible drift however comes with a price: The
right hand side of the SDE \eqref{eq:irrevSDE} becomes rather stiff as
$\veps \to 0$, and as a result, standard integrators (for example the
Euler-Maruyama scheme) would require vanishingly small time step size
to resolve the fast scale of the dynamics. As $\veps$ goes to zero,
the SDE contains multiple time scale, and thus it is better to use multiscale integrators for such dynamics.

In this work, we investigate a multiscale integrator for stiff SDEs as
\eqref{eq:irrevSDE} proposed in \cite{TaoOwhadiMarsden2010}, which is
also rather close to the seamless version of HMM scheme
\cite{ELiuVE:05, ELu:07, ERenVE:09}. For a macro time step $\delta$
and micro time step $\tau$ such that $\tau \ll \veps \ll \delta$, from
$t_n$ to $t_n + \delta$, we evolve the dynamics
\begin{subequations}
\begin{align}
  & \ud \wb{Z}_t = \bigl[ - \nabla U(\wb{Z}_t) + \frac{1}{\veps} C(\wb{Z}_t)
  \bigr] \ud t + \sqrt{2 \beta} \ud W_t &&  \qquad  t \in [t_n, t_n + \tau); \\
  & \ud \wb{Z}_t = - \nabla U(\wb{Z}_t)  \ud t + \sqrt{2 \beta} \ud W_t &&  \qquad  t \in [t_n + \tau, t_n + \delta).
\end{align}
\end{subequations}

This can be understood as a split-step time integrator where for the
short time step $\tau$ we use the whole SDE and for the long time step
$\delta - \tau$ we neglect the irreversible drift.  The equations
above can be integrated using standard numerical schemes, and for
definiteness, in this work we will discretize using the standard
Euler-Maruyama method, which gives
\begin{subequations}\label{eq:hmmintegrator}
\begin{align}
  & \wb{Z}_{t_n + \tau} - \wb{Z}_{t_n} = - \tau \nabla U (\wb{Z}_{t_n}) + \frac{\tau}{\veps} C(\wb{Z}_{t_n}) + \sqrt{2 \beta \tau} \, \xi_n; \\
  & \wb{Z}_{t_n + \delta} - \wb{Z}_{t_n + \tau} = - (\delta - \tau) \nabla U(\wb{Z}_{t_n + \tau}) + \sqrt{2 \beta (\delta - \tau)} \, \xi'_n,
\end{align}
\end{subequations}
where $\xi_n$ and $\xi'_n$ are independent standard normal random
variables.

As will be discussed in Sections \ref{sec:averaging} and
  \ref{sec:analysis}, if the dynamical system $\dot{z}_{t}=C(z_{t})$
  does not have a unique invariant measure on each connected component
  of the level sets of $U(z)$ and the dimension is bigger than two,
  then one needs to modify the scheme by considering an additional
  regularizing noise, see Condition \ref{A:Assumption2}. In particular, in this case we may need to regularize the problem by introducing  an additional artificial noise component in the fast dynamics, i.e.,
\begin{equation}\label{Eq:RegularizedIrreversibleLangevin0}
dZ^{\epsilon}_{t}=\left[-\nabla U(Z^{\epsilon}_{t})dt+\sqrt{2\beta}dW_{t}\right]+ \left[ \frac{1}{\epsilon}  \wt{C}(Z^{\epsilon}_{t})dt+\sqrt{\frac{\kappa}{\epsilon}} \sigma(Z^{\epsilon}_{t})dW^{o}_{t}\right].
\end{equation}
Here, $W$ and $W^{o}$ are independent standard Wiener processes, the
matrix $\sigma$ will be specified in Condition \ref{A:Assumption4} below, and we have defined
\begin{equation}
\wt{C}_i(z)=C_i(z)+\frac{\kappa}{2}\sum_{j=1}^{d}\frac{\partial \left[\sigma\sigma^{T}(z)\right]_{j,i}}{\partial z_{j}}, \qquad i = 1, \ldots, d. \label{Eq:RegularizedIrreversibleLangevin1}
\end{equation}

If $\kappa=0$ then the fast motion is the deterministic dynamical system $\dot{z}_{t}=C(z_{t})$ and
$Z^{\epsilon}_{t}$ is  a random perturbation of this dynamical system. For example, if $d$ is even we can take $C$ to be the
Hamiltonian vector field $C(z)=J \nabla U(z)$. If $\kappa>0$ we have random perturbations of diffusion processes with a
conservation law. The artificial perturbation is chosen such that (\ref{Eq:RegularizedIrreversibleLangevin0}) still samples from the same Boltzmann-Gibbs measure $\rho(z)\propto e^{-U(z)/\beta}$. In addition, we emphasize here that the limiting behavior as $\epsilon\rightarrow 0$ is not affected by the additional regularizing noise, since neither $\kappa$ nor $\sigma(z)$  appear in the limiting dynamics, see Theorem \ref{T:ConvNumericalScheme}. Additionally, as it is proven in \cite{ReyBelletSpiliopoulos2016} adding such perturbations does not make the performance worse in terms of all three criteria, spectral gap, asymptotic variance and large deviations rate function. 

As we shall also see in \eqref{Eq:HMMscheme} in
  Section~\ref{sec:analysis}, in the case of the perturbation
  \eqref{Eq:RegularizedIrreversibleLangevin0}, the algorithm naturally
  extends to the form
\begin{subequations}\label{eq:hmmintegratorRegularized}
\begin{align}
  & \wb{Z}_{t_n + \tau} - \wb{Z}_{t_n} = - \tau \nabla U (\wb{Z}_{t_n}) + \frac{\tau}{\veps} \wt{C}(\wb{Z}_{t_n}) + \sqrt{2 \beta \tau} \, \xi_n + \sqrt{\frac{\tau}{\epsilon}}\sqrt{\kappa}\sigma(\wb{Z}_{t_n})\, \xi^{''}_n; \\
  & \wb{Z}_{t_n + \delta} - \wb{Z}_{t_n + \tau} = - (\delta - \tau) \nabla U(\wb{Z}_{t_n + \tau}) + \sqrt{2 \beta (\delta - \tau)} \, \xi'_n,
\end{align}
\end{subequations}
where $\xi_n$, $\xi'_n$, $\xi^{''}_n$ are independent standard normal random variables.

We will show that with proper choices of the time steps $\tau$ and
$\delta$ as $\veps \to 0$, the numerical schemes
\eqref{eq:hmmintegrator}, and more generally \eqref{eq:hmmintegratorRegularized}, converge to the diffusion on graphs, which
is the averaging limit of \eqref{eq:irrevSDE}. Thus, we may use
\eqref{eq:hmmintegrator}, or more generally \eqref{eq:hmmintegratorRegularized}, to numerically discretize the SDE which is
consistent in the asymptotic regime as $\veps \to 0$.

Let us remark that it is also possible to use multiple micro steps
with length $\tau$ rather than just one such step as in
\eqref{eq:hmmintegrator}, which is analogous to the original HMM
integrators.  For the purpose of sampling invariant measure, one could
also combine the integrator with Metropolis adjustment steps as in the
MALA method \cite{RobertsTweedie:96}, see
  \cite{OttobrePillaiSpiliopoulos2017} for some preliminary results
  towards this direction. We will focus on the numerical analysis of
the scheme \eqref{eq:hmmintegrator}, and more generally
\eqref{eq:hmmintegratorRegularized}, and leave these extensions to
future works.

\section{Numerical examples}\label{sec:numerics}

Before we turn to the analytical results, let us present a few
numerical tests for the multiscale integrator. We will first consider
the sampling efficiency of the irreversible Langevin sampler with the
multiscale HMM integrator. We will then show some numerical examples
illustrating properties of the integrator. We limit ourselves to
simple toy examples as the focus is to demonstrate the numerical
properties of the integrator and validate the numerical analysis
results, rather than applying to scheme to realistic problems.

For the first test, we consider a $2D$ symmetric double well potential given by
\begin{equation}\label{eq:doublewell}
  U(x, y) = \frac{1}{4} (x^2 - 1)^2 + y^2,
\end{equation}
with inverse temperature $\beta = 0.1$. We choose $C = J \nabla U$ with $J =
\bigl( \begin{smallmatrix} 0 & 1 \\ -1 & 0
\end{smallmatrix} \bigr)$. The initial condition is set to be the
origin, and we consider the empirical average of the observable $f(x,
y) = x + y^2$ over the time interval $[0, 2000]$ with a burn-in period
$T_{\text{burn}} = 20$. To test the performance of the sampling scheme
based on the multiscale integrator, we compare the empirical average
with the true average of the observable with respect to the invariant
measure. Note that in this case, due to the choice of the potential
and the observable, the true average is explicitly given by
$\average{f} = \frac{1}{2} \beta = 0.05$.  We also estimate the
asymptotic variance of the sampling scheme by dividing the sampled
data points into $20$ batches.

Denote the sampling error (with respect to the observable $f$) as
$\text{Err}_f$ and the asymptotic variance (with respect to the
observable $f$) as $\text{AVar}_f$. The numerical results for various choice of
$\veps$ are shown in Table~\ref{tab:doublewell}, in which we also
include the results for direct Euler-Maruyama discretization for
comparison. In these tests, we fix the macro time step $\delta =
5e$-$3$ (for the direct Euler-Maruyama discretization, $\delta$ is the
time step size), and choose the micro time step $\tau = 10 \delta
\veps = 0.05 \veps$. For a given set of parameters, we report the mean
and standard deviation estimated from $20000$ independent runs of the
algorithms. We note that the Euler-Maruyama scheme is unstable for
$\veps$ below $0.1$ with the given $\delta$.


\begin{table}[h]
\begin{tabular}{c|c|c|c|cc|cc}
  \hline
  & $\veps$ & $\tau$ & $\delta$ & $\EE(\text{Err}_f)$ & $\std(\text{Err}_f)$ &
                                                                               $\EE(\text{AVar}_f)$ & $\std(\text{AVar}_f)$ \\
  \hline
  E-M & $\infty$ & & $5e$-$3$ & $1.2573e$-$1$ & $9.4758e$-$2$ & $3.4865e$-$1$ & $7.5354e$-$2$  \\
  & $5$ & & $5e$-$3$ & $1.2430e$-$1$ & $9.3051e$-$2$ & $3.4243e$-$1$ & $7.3883e$-$2$ \\
  & $5e$-$1$ & & $5e$-$3$ & $7.2015e$-$2$ & $5.4729e$-$2$ & $1.4527e$-$1$ & $4.0925e$-$2$ \\
  & $1e$-$1$ & & $5e$-$3$ & $4.0662$-$2$ & $2.5206e$-$2$ & $1.7146e$-$2$ & $5.4579e$-$3$ \\
  \hline
  HMM & $1e$-$2$ & $5e$-$4$ & $5e$-$3$ & $3.9735e$-$2$ & $2.5167e$-$2$ & $1.7344e$-$2$ & $5.5658e$-$3$ \\
  & $1e$-$3$ & $5e$-$5$ & $5e$-$3$ & $3.9568e$-$2$ & $2.5404e$-$2$ & $1.7253e$-$2$ & $5.5275e$-$3$ \\
  & $1e$-$4$ & $5e$-$6$ & $5e$-$3$ & $3.9416e$-$2$ & $2.5374e$-$2$ & $1.7325e$-$2$ & $5.5570e$-$3$ \\
  \hline
\end{tabular}
\medskip
\caption{Comparison of the Euler-Maruyama scheme and the HMM multiscale integrator for the double well potential \eqref{eq:doublewell}. The macro time step is $\delta = 5e$-$3$ with the micro time step $\tau = 0.05 \veps$. The mean and standard deviation of the sampling error $\text{Err}_f$ and asymptotic variance $\text{AVar}_f$ are reported for various choice of $\veps$. The case $\veps = \infty$ means sampling without adding the irreversible drift.
  The stability threshold for the Euler-Maruyama scheme for the time step size $\delta = 5e$-$3$ is around $\veps = 8.25e$-$2$.}\label{tab:doublewell}
\end{table}


\begin{table}
  \begin{tabular}{c|c|c|c|cc|cc}
    \hline
    & $\veps$ & $\tau$ & $\delta$ & $\EE(\text{Err}_f)$ & $\std(\text{Err}_f)$  & $\EE(\text{AVar}_f)$ & $\std(\text{AVar}_f)$ \\
    \hline
    E-M & $5e$-$2$ & & $1e$-$3$ & $3.1930e$-$2$ & $2.2377e$-$2$ & $1.6638e$-$2$ & $5.2778e$-$3$ \\
    \hline
    HMM  & $1e$-$3$ & $2e$-$5$ & $1e$-$3$ & $3.1751e$-$2$ & $2.2387e$-$2$ & $1.6637e$-$2$ & $5.3156e$-$3$ \\
    & $1e$-$4$ & $2e$-$6$ & $1e$-$3$ & $3.1886e$-$2$ & $2.2319e$-$2$ & $1.6702e$-$2$ & $5.2691e$-$3$ \\
    & $1e$-$5$ & $2e$-$7$ & $1e$-$3$ &  $3.1790e$-$2$ & $2.2361e$-$2$ & $1.6611e$-$2$ & $5.3201e$-$3$ \\
    \hline
    HMM & $1e$-$4$ & $1e$-$6$ & $1e$-$3$ & $3.5224e$-$2$ & $2.6448e$-$2$ & $3.7264e$-$2$ & $1.1721e$-$2$\\
    & $1e$-$5$ & $1e$-$7$ & $1e$-$3$ & $3.5207e$-$2$ & $2.6483e$-$2$ &
$3.7391e$-$2$ & $1.1698e$-$2$ \\
    \hline
\end{tabular}
\smallskip
\caption{Comparison of the Euler-Maruyama scheme and the HMM multiscale integrator for the double well potential \eqref{eq:doublewell} with reduced  macro time step $\delta = 1e$-$3$ (compared with $\delta = 5e$-$3$ in Table~\ref{tab:doublewell}). The Euler-Maruyama scheme is now stable with $\veps = 5e$-$2$; and the stability threshold is around $\veps = 3.8e$-$2$. The micro time step in HMM scheme is chosen to be either $\tau = 20 \delta \veps =  0.02 \veps$ or $\tau = 10 \delta \veps = 0.01 \veps$. The mean and standard deviation of the sampling error $\text{Err}_f$ and asymptotic variance $\text{AVar}_f$ are reported for various choice of $\veps$.
} \label{tab:doublewell2}
\end{table}

We remark that the particular choice of the observable $f = x + y^2$
makes the accurate sampling rather challenging in this case: As $\EE x
= 0$ due to the symmetry, the correct sampling of the average value
requires fine balance of the time the trajectory spent in left and
right well of the double well potential. This explains the high
relative error that $\EE(\text{Err}_f)$ is on the same order of $\EE
f$.

We make several observations in regards to the numerical results in
Table~\ref{tab:doublewell}. First, from the result of the
Euler-Maruyama scheme for various $\veps$, it is clear that a larger
irreversible drift (smaller $\veps$) enhances the sampling as the
sampling error and also the asymptotic variance decrease. Second, while for a fixed computational cost the Euler-Maruyama scheme
becomes unstable for small $\veps$, the HMM scheme works well for
smaller $\veps$ which further reduces the sampling error. Moreover, we
remark that while the integrator works well for very small $\veps$,
the improvement in this example for going to a very small $\veps$ is
limited, this is expected since when $\veps \to 0$, as will be shown
later, the scheme becomes an approximation of the averaging limit of
the SDE. Thus the sampling efficiency is determined by the limiting
system, and the impact of a finite but small $\veps$ may be
negligible. Of course, this depends on how fast ergodicity kicks in allowing the averaging limit to be achieved. Note that the HMM multiscale integrator allowed us to reach to the $\veps \to 0$ limit stably, while the Euler-Maruyama scheme blows up for small values of $\veps$ keeping $\delta$ fixed.

We further test the dependence of the HMM scheme on the choice of
parameters in Table~\ref{tab:doublewell2}. In those tests, we decrease
the value of macro time step to $\delta = 1e$-$3$. In comparison, we
also list the result of the Euler-Maruyama scheme with the same time
step, which is now stable for smaller $\veps = 5e$-$2$ (but loses
stability if we further reduce $\veps$). The results for the HMM
scheme with different $\veps$ and $\tau = 20\delta \veps = 0.02\veps$
suggest that it is better to take a smaller $\veps$, though the
improvement is again marginal in this case. Compared with
Table~\ref{tab:doublewell}, we see that a smaller $\delta$ improves
the sampling results, though of course this comes with a higher
computational cost.

In Table~\ref{tab:doublewell2}, we also consider choice of the micro
time step $\tau$ with different ratios of $\tau / (\delta \veps)$ to
see the dependence.  We observe that in the case with the smaller
$\delta$, if we still take $\tau$ such that $\tau = 10 \delta \veps$
as in Table~\ref{tab:doublewell}, the performance of the sampling
scheme is in fact worse than the direct Euler-Maruyama scheme (with a
larger $\veps$). Thus it motivates the choice of a larger $\tau$,
which increases the effective sampling time of the fast dynamics, and
hence is expected to lead to better performance. This is confirmed in
the numerical results with the choice of
$\tau = 20 \delta \veps = 0.02\veps$. Unfortunately, if we further
increase $\tau$ (choosing for example
$\tau = 30 \delta \veps = 0.03\veps$ here), the HMM scheme becomes
unstable. This instability can be understood in our theoretical
analysis as the assumption that $(\tau / \veps)^{3/2} \ll \delta$ in
the convergence result Theorem~\ref{T:ConvNumericalScheme}.
With a fixed macro time step $\delta$ (and hence
fixing computational budget), it seems that a good practice is to
choose a larger ratio $\tau / (\veps \delta)$ while making sure that
the scheme is stable.

\medskip

In the next example, we consider a more complicated potential, still
in two dimension, given by
\begin{equation}\label{eq:RBS3}
  U(x, y) = \frac{1}{4} \Bigl[ (x^2 -1)^2 ((y^2 - 2)^2 + 1) + 2 y^2 - y / 8 \Bigr] + e^{-8x^2 - 4 y^2}
\end{equation}
with inverse temperature $\beta = 0.2$. This is the potential
considered in \cite[Example 3]{ReyBelletSpiliopoulos2015}. We take the
observable $f(x, y) = (x-1)^2 + y^2$ and setting $\delta = 5e$-$3$ and
$\tau = 10 \delta \veps = 0.05 \veps$. The total simulation time is $T
= 2000$ with a burn-in period $T_{\text{burn}} = 20$. $20000$
independent runs of the algorithms are used to get statistics of the
sampling error and asymptotic variance. The results are reported in
Table~\ref{tab:RBS3}. Here the true average of the observable is
obtained by a discretization of the Gibbs distribution on the phase
space with a fine mesh, which gives approximately $\EE f \approx
2.1986$.  Similarly as in the double well potential the Euler-Maruyama
scheme loses stability for $\veps$ smaller than $0.1$. The conclusion
of the numerical results is similar to that of the double well
example.


\begin{table}[h]
\begin{tabular}{c|c|c|c|cc|cc}
  \hline
  & $\veps$ & $\tau$ & $\delta$ & $\EE(\text{Err}_f)$ & $\std(\text{Err}_f)$ &
  $\EE(\text{AVar}_f)$ & $\std(\text{AVar}_f)$ \\
  \hline
  E-M  & $\infty$ & & $5e$-$3$ & $4.0022e$-$1$ & $3.6110e$-$1$ & $2.5204e00$ & $5.3007e$-$1$  \\
 & $5$ & & $5e$-$3$ & $4.9166e$-$1$ & $3.5878e$-$1$ & $2.5132e00$ & $5.2312e$-$1$ \\
  & $5e$-$1$ & & $5e$-$3$ & $3.1675e$-$1$ & $2.3714e$-$1$ & $1.8778e00$ & $3.6895e$-$1$ \\
  & $1e$-$1$ & & $5e$-$3$  & $1.0716e$-$1$ & $8.1086e$-$2$ & $3.4154e$-$1$ & $1.0282e$-$1$ \\
  & $6e$-$2$ & & $5e$-$3$ & $1.7342e$-$1$ & $5.6951e$-$2$ & $6.5024e$-$2$ & $2.0713e$-$2$ \\
  \hline
  HMM & $1e$-$2$ & $5e$-$4$ & $5e$-$3$ & $1.0283e$-$1$ & $7.6923e$-$2$ & $3.0595e$-$1$ & $9.3185e$-$2$ \\
  & $1e$-$3$ & $5e$-$5$ & $5e$-$3$ & $1.0107e$-$1$ & $7.5928e$-$2$ & $3.0000e$-$1$ & $9.1497e$-$2$ \\
  & $1e$-$4$ & $5e$-$6$ & $5e$-$3$ & $1.0216e$-$1$ & $7.7365e$-$2$ & $2.9930e$-$1$ & $9.0837e$-$2$ \\
  \hline
\end{tabular}
\smallskip
\caption{Comparison of the Euler-Maruyama scheme and the HMM multiscale integrator for the potential \eqref{eq:RBS3}. The macro time step is $\delta = 5e$-$3$ and the micro time step $\tau = 0.05 \veps$. The mean and standard deviation of the sampling error $\text{Err}_f$ and asymptotic variance $\text{AVar}_f$ are reported for various choice of $\veps$. The case $\veps = \infty$ means sampling without adding the irreversible drift.
  The stability threshold for the Euler-Maruyama scheme  is around $\veps = 6e$-$2$, which is included in the above table.}\label{tab:RBS3}
\end{table}

Next we plot the $x$ coordinate of a sample trajectory and the
corresponding potential energy $U(x(t), y(t))$ for the double well
potential in Figure~\ref{fig:traj}. The simulation is done with $\veps
= 1e$-$5$ and time step sizes $\tau = 5e$-$7$ and $\delta = 0.05$. The
trajectory is plotted at the end of every macro step (so on the
interval of $\delta$). The plot focuses on the time period $[20, 21]$
during which the trajectory mainly stays in the left portion of the
phase space $\{x < 0\}$. As can be clearly observed from the figure,
while the solution of the SDE oscillates very fast, the potential
energy $U$ changes much more slowly, which suggests that $U$ is a slow
variable in the averaging limit.

\begin{figure}[ht]
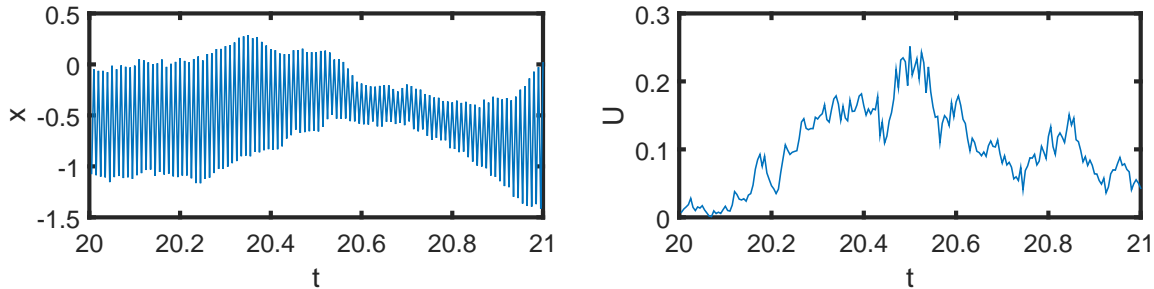

  \includegraphics[width = .45\textwidth]{x1hist.eps}
  \quad
  \includegraphics[width = .45\textwidth]{Uhist.eps}
  \caption{A sample trajectory of the HMM multiscale integrator for
    the double well potential \eqref{eq:doublewell} with $\veps =
    1e$-$5$, $\tau = 5e$-$7$, and $\delta = 0.05$. (Left) The $x$
    cooridnate of the trajectory for $t \in [20, 21]$. (Right) The
    potential energy $U$ associated with the
    trajectory.}\label{fig:traj}
\end{figure}

It is clear that the map from $(x, y)$ to $U$ is not one-to-one in
these examples. For the double well potential, below the energy of the
saddle point $U(0, 0) = \frac{1}{4}$, the isopotential curve is
disjoint and separated into the left and right half-planes
corresponding to the two minima of the potential at $(-1, 0)$ and $(1,
0)$, which are represented on the corresponding graph of the potential
as two separated edges that meet at the interior vertex $U =
\frac{1}{4}$ corresponding to the saddle point (see
section~\ref{sec:averaging} where these concepts are
recalled). Therefore, to get from one component of the isopotential
curve to the other, the trajectory has to go up in energy and cross
the interior vertex; when the energy is decreased from above
$\frac{1}{4}$, the trajectory would go into one of the edges,
corresponding to one of the disjoint components. In fact, such an
event of energy goes above the saddle point energy can be observed
already in Figure~\ref{fig:traj} around $t = 20.5$, where the energy
first goes up and when it drops down, the trajectory goes back to the
same potential well (on the left half-plane).

To see how the multiscale integrator captures diffusion across the
interior vertices on the graph, we record the number of transitions to
each edge with lower energy connecting to the saddle point when the
energy of the trajectory is decreasing from above that of the
vertex. For the double well potential, we count the transitions into
each component for a long trajectory with total time $T = 2000$ (with
burn-in time $T_{\text{burn}} = 20$) when the energy decreases from
above $\frac{1}{4}$. The simulation parameters are $\veps = 1e$-$5$,
$\tau = 5e$-$7$ and $\delta = 0.005$.  For a single realization of the
algorithm, we obtain
\begin{equation*}
  N_{\text{left}} = 3719, \quad \text{and} \quad N_{\text{right}} = 3659,
\end{equation*}
where $N_{\text{left}}$ denotes the number of times the trajectory
goes to the left well and $N_{\text{right}}$ for the right well during
the time period $[T_{\text{burn}}, T]$. Note that the empirical
probability of going to the left well is $0.5041$, very close to the
theoretical value of the diffusion on the graphs (which is $0.5$ due
to the symmetry). While the data reported is only for one realization,
this is the typical behavior observed for the algorithm.

In comparison, let us now consider a similar test for a tilted double well potential (so that the symmetry is broken):
\begin{equation}
  U(x, y) =   \frac{1}{4} (x^2 - 1)^2 - \frac{1}{8} x + y^2.\nonumber
\end{equation}
The tilting by $-\frac{1}{8} x$ moves the local minima and the saddle
point of the potential to approximately $(-0.9304, 0)$, $(1.0575, 0)$,
and $(-0.12705, 0)$. We repeat the same calculation as in the
symmetric double well case and obtain
\begin{equation*}
  N_{\text{left}} = 2485, \quad \text{and} \quad N_{\text{right}} = 3704,
\end{equation*}
so that the empirical probability of going to the left well is
$0.4015$. Due to the asymmetry, it is less likely to go to the left
well when the energy decreases from above the value of the saddle
point. This is consistent with the theoretical results for the multiple well case that we will establish in section~\ref{sec:analysis}.

\section{The averaging problem}\label{sec:averaging}

The irreversible perturbations with a small $\epsilon$ induce a fast motion on the constant potential surface and
slow motion in the orthogonal direction. Using the theory of diffusions on  graphs and the related averaging principle, see  \cite{BrinFreidlin2000,FreidlinWeber2004,FreidlinWentzell1993}, we may identify the limiting motion of the slow component, see \cite{ReyBelletSpiliopoulos2015}. The fast motion on constant potential surfaces decreases the variance of the estimator as the phase space is explored more efficiently.
Let us consider the level set
\begin{equation}
d(x)=\left\{z\in E: U(z)=x\right\},\nonumber
\end{equation}
where $E\subseteq\mathbb{R}^{d}$ denotes the state space.
We then denote by $d_{i}(x)$  the connected components of $d(x)$, i.e.,
\begin{equation}
d(x)=\bigcup_{i}d_{i}(x)\,.\nonumber
\end{equation}

We define  $\Gamma$ to be the graph which is homeomorphic to the set of connected components $d_{i}(x)$ of the level sets $d(x)$.
Exterior vertexes correspond to minima of $U$, whereas interior vertexes correspond to saddle points of $U$. The edges of $\Gamma$
are indexed by $I_{1},\cdots, I_{m}$. Each point on $\Gamma$ is indexed by a pair $y=(x,i)$ where $x$ is the value of $U$ on the level
set corresponding to $y$ and $i$ is the edge number containing $y$. Clearly the pair $y=(x,i)$
forms a global coordinate on $\Gamma$. Let $Q: E \mapsto\Gamma$ with $Q(z)=(U(z),i(z))$
be the corresponding projection on the graph. For an edge $I_{k}$ and a vertex $O_{j}$ we write $I_{k}\sim O_{j}$ if $O_{j}$
lies at the boundary of the edge $I_{k}$. We endow the tree $\Gamma$ with the natural topology. It is known that $\Gamma$ forms a
graph  with interior vertexes of order  two or three, see for example \cite{FW2}.

If the dynamical system $\dot{z}_{t}=C(z_{t})$ does not
have a unique invariant measure on each connected component $d_{i}(x)$, then we may need to regularize the problem by introducing  an additional artificial noise component in the fast dynamics, i.e.,
\begin{equation}\label{Eq:RegularizedIrreversibleLangevin}
dZ^{\epsilon}_{t}=\left[-\nabla U(Z^{\epsilon}_{t})dt+\sqrt{2\beta}dW_{t}\right]+ \left[ \frac{1}{\epsilon}  \wt{C}(Z^{\epsilon}_{t})dt+\sqrt{\frac{\kappa}{\epsilon}} \sigma(Z^{\epsilon}_{t})dW^{o}_{t}\right].
\end{equation}
Here, $W$ and $W^{o}$ are independent standard Wiener processes, the
matrix $\sigma$ will be specified below, and we have defined
\[
\wt{C}_i(z)=C_i(z)+\frac{\kappa}{2}\sum_{j=1}^{d}\frac{\partial \left[\sigma\sigma^{T}(z)\right]_{j,i}}{\partial z_{j}}, \qquad i = 1, \ldots, d.
\]

If $\kappa=0$ then the fast motion is the deterministic dynamical system $\dot{z}_{t}=C(z_{t})$ and
$Z^{\epsilon}_{t}$ is  a random perturbation of this dynamical system. For example, if $d$ is even we can take $C$ to be the
Hamiltonian vector field $C(z)=J \nabla U(z)$. If $\kappa>0$ we have random perturbations of diffusion processes with a
conservation law.

We make several technical assumptions on $C(z)$, $U(z)$ and
$\sigma(z)$ in order to guarantee that the averaging principle applies
to \eqref{Eq:RegularizedIrreversibleLangevin}. We make these
assumptions in order to guarantee that the fast process has a unique
invariant measure and will have $U$ as a smooth first integral.

Let us next identify the corresponding fast and slow components. The fast motion corresponds to the infinitesimal generator
\begin{equation}
\widehat{\mathcal{L}}g(z)=\wt{C}(z)\nabla g(z)+\frac{\kappa}{2}\text{tr}\left[\sigma\sigma^{T}(z)D^{2}g(z)\right]= C(z)\nabla g(z)+\frac{\kappa}{2}\nabla \left[\sigma\sigma^{T}(z)\nabla g(z)\right]\,.\label{Eq:FastOperator}
\end{equation}

Let us write $\widehat{Z}_{t}$ for the diffusion process that has infinitesimal generator $\widehat{\mathcal{L}}$. In
order to guarantee the existence of a unique invariant measure for the
fast dynamics we assume:
\begin{condition}\label{A:Assumption2}
  In dimension $d=2$, we take $\kappa\geq 0$. In dimension $d>2$, we
  either assume that the dynamical system $\dot{z}_{t}=C(z_{t})$ has a
  unique invariant measure on each connected component $d_{i}(x)$, in
  which case we take $\kappa\geq 0$, or otherwise we assume that
  $\kappa>0$.
\end{condition}

As far as the potential function $U(z)$ and the perturbation $C(z)$ are concerned, we shall assume Condition \ref{A:Assumption3}. Condition \ref{A:Assumption3} guarantees that $C(z)$ does not affect the invariant measure of the process and it imposes natural growth and structural conditions on the potential $U$.
\begin{condition}\label{A:Assumption3}
The potential function $U(z)$ and the perturbation $C(z)$ satisfy
\begin{enumerate}
\item{There exists $a>0$ such that $U\in\mathcal{C}^{(2+a)}(E)$ and $C\in \mathcal{C}^{(1+a)}(E)$.}
\item{ $U(z)\geq A_{1}|z|^{2}$, $\nabla U(z)\geq A_{2}|z|$ and $\Delta U(z)\geq A_{3}$ for sufficiently large $|z|$, where $A_{1},A_{2},A_{3}$ are positive constants.}
\item{$\divop C(z)=0$ and $C(z)\cdot\nabla U(z)=0$.}
\item{$U$ has a finite number of critical points $z_{1},\ldots,z_{m}$
    and at these points the Hessian matrix $D^2 U$ is
    non-degenerate.}
\item{There is at most one critical point for each connected level set component of $U$.}
\item{If $z_{k}$ is a critical point of $U$, then there exists a
    constant $d_{k}>0$ such that $C(z)\leq d_{k}|z-z_{k}|$.}
\item{If $d=2$ and $\kappa=0$, then $C(z)=0$ implies $\nabla U(z)=0$ and for any saddle point $z_{k}$ of $U(z)$, there exists a constant $c_{k}>0$ such that $|C(z)|\geq c_{k}|z-z_{k}|$.}
\end{enumerate}
\end{condition}

In regards to the additional artificial perturbation by the noise $W^{o}_{t}$, i.e., when $\kappa>0$, we assume Condition \ref{A:Assumption4}. Condition \ref{A:Assumption4} guarantees that the extra regularization with the noise does not affect the invariant measure and that it is such that the subsequent averaging analysis goes through, see \cite{FreidlinWeber2004}.
\begin{condition}\label{A:Assumption4}
\begin{enumerate}
\item{The matrix $\sigma(z)\sigma^{T}(z)$ is symmetric, non-negative
    definite, and with smooth entries.}
\item{$\sigma(z)\sigma^{T}(z)\nabla U(z)=0$ for all $z\in E$.}
\item{For any $z\in E$ and $\xi \in \mathbb{R}^d$ such that
    $\xi\cdot\nabla U(z)=0$ we have that
    $\lambda_{1}(z)|\xi|^{2}\leq \xi^{T}\sigma(z)\sigma^{T}(z)\xi\leq
    \lambda_{2}(z)|\xi|^{2}$
    where $\lambda_{1}(z)>0$ if $\nabla U(z)\neq 0$ and there exists a
    constant $K$ such that $\lambda_{2}(z)<K$ for all $z\in E$.
Moreover if $z_{k}$ is a critical point for $U$, then there are positive constants $k_{1},k_{2}$ such that for all $z$ in a neighborhood of $z_{k}$
    \begin{equation*}
    \lambda_{1}(z)\geq k_{1} |z-z_{k}|^{2}, \text{ and }     \lambda_{2}(z)\leq k_{2} |z-z_{k}|^{2}.
    \end{equation*}}
\item{Let $\lambda_{i,k}$ be the eigenvalues of the Hessian of $U(z)$ at the critical points $z_{k}$ where $k=1,\cdots,m$ and $i=1,\cdots, d$. Then we assume that
$0<\kappa\ll\left(K\max_{i,k}\lambda_{i,k}\right)^{-1}$.}
\end{enumerate}
\end{condition}

We remark here that the end result does not depend on the additional regularizing noise, since $\sigma(z)$ does not appear in the limiting dynamics.
Recall that $\widehat{Z}_{t}$ is the diffusion process that has infinitesimal generator $\widehat{\mathcal{L}}$. Conditions \ref{A:Assumption3} and \ref{A:Assumption4} guarantee that with probability one, if the initial point of $\widehat{Z}$ is in a connected component $d_{i}(x)$, then $\widehat{Z}_{t}\in d_{i}(x)$ for all $t\geq 0$. Indeed, by It\^{o} formula we have
\begin{equation*}
U(\widehat{Z}_{t})=U(\widehat{Z}_{0})+\int_{0}^{t}\widehat{\mathcal{L}}U(\widehat{Z}_{s})ds+\int_{0}^{t}\nabla U(\widehat{Z}_{s})\sigma(\widehat{Z}_{s})dW_{s}.
\end{equation*}
Since $C(z)\nabla U(z)=0$ and $\sigma(z)\sigma^{T}(z)\nabla U(z)=0$ we obtain with probability one $\int_{0}^{t}\widehat{\mathcal{L}} U(\widehat{Z}_{s})ds=0$. The quadratic variation of the stochastic integral is also zero, due to $\sigma(z)\sigma^{T}(z)\nabla U(z)=0$, which implies that with probability one  $\int_{0}^{t}\nabla U(\widehat{Z}_{s})\sigma(\widehat{Z}_{s})dW_{s}=0$. Thus, we indeed get that for all $t\geq 0$ $\widehat{Z}_{t}\in d_{i}(x)$ given that the initial point belongs to the particular connected component of the level set $d_{i}(x)$. In particular,  It\^{o} formula gives for a test function $f\in\mathcal{C}^{2}(\mathbb{R})$
\begin{align}
f(U(Z^{\veps}_{t}))&=f(U(z))+\int_{0}^{t}\left(\Bigl[-|\nabla U(Z^{\veps}_{s})|^{2}+\beta \text{tr}\left[D^{2}U(Z^{\veps}_{s})\right]\Bigr]f^{'}(U(Z^{\veps}_{s}))+\beta |\nabla U(Z^{\veps}_{s})|^{2} f^{''}(U(Z^{\veps}_{s}))\right)ds\nonumber\\
&\quad+\sqrt{2\beta}\int_{0}^{t}f^{'}(U(Z^{\veps}_{s}))\nabla U(Z^{\veps}_{s})dW_{s}\nonumber
\end{align}
where $Z^{\veps}$ is the solution to
\eqref{Eq:RegularizedIrreversibleLangevin}.

Let $m(z)$ be a smooth invariant density with respect to Lebesgue measure for the
process $\widehat{Z}_{t}$. The fact that $m(z)$ exists and is smooth follows from the discussion in Section 2.3 of \cite{FreidlinWeber2004}. Then, the proof of \cite[Lemma 2.3]{FreidlinWeber2004} and the fact that $t\geq 0$ $\widehat{Z}_{t}\in d_{i}(x)$ if $
\widehat{Z}_{0}\in d_{i}(x)$ imply that if $(x,i)\in\Gamma$ is not a vertex, there exists a unique invariant measure $\mu_{x,i}$ concentrated on
the connected component $d_{i}(x)$ of $d(x)$ which takes the form
\begin{equation*}
\mu_{x,i}(A)=\frac{1}{T_{i}(x)}\oint_{A}\frac{m(z)}{\left|\nabla U(z) \right|}\ell(dz) \,,
\end{equation*}
where $\ell(dz)$ is the surface measure on $d_i(x)$ and
$T_{i}(x)=\oint_{d_{i}(x)}\frac{m(z)}{\left|\nabla U(z)
  \right|}\ell(dz)$. 
Notice that if $(x,i)\in\Gamma$ is not a vertex, then the invariant
density on $d_{i}(x)$ is
\begin{equation*}
m_{x,i}(z)=\frac{m(z)}{T_{i}(x)\left|\nabla U(z) \right|},\quad z\in d_{i}(x).
\end{equation*}

We remark here that in the case $\kappa>0$, it is relatively easy to see that independently of the form of the matrix $\sigma(z)\sigma^{T}(z)$, the fact that $\divop(C)=0$ implies that the Lebesgue measure is invariant for the diffusion process corresponding to the operator
$\widehat{\mathcal{L}}$. Hence, in that case any constant function is an invariant density. Also, in the case $d=2$ and $\kappa=0$, one immediately obtains from Condition \ref{A:Assumption3} that $m(z)=\frac{|\nabla U(z)|}{|C(z)|}$, see  \cite[Proposition 2.1]{FreidlinWeber2004}.

Given a sufficiently smooth  function $f(z)$, define its average over the related connected component of the level set of $U(z)$ by
\begin{equation*}
\widehat{f}(x,i)=\oint_{d_{i}(x)}f(z)m_{x,i}(z)\ell(dz)=\frac{1}{T_{i}(x)}\oint_{d_{i}(x)}\frac{f(z)}{\left|\nabla U(z) \right|}m(z)\ell(dz).\label{Eq:AverageOnGraph}
\end{equation*}

We write $\mathcal{L}_{0}$ for the infinitesimal generator of the process $Z_{t}$ given by (\ref{eq:irrevSDE}) with $C(z)=0$. Let us then set
\begin{align*}
\widehat{\mathcal{L}_{0}U}(x,i)&=\oint_{d_{i}(x)}\mathcal{L}_{0}U(z)m_{x,i}(z)\ell(dz)=\frac{1}{T_{i}(x)}\oint_{d_{i}(x)}\frac{\mathcal{L}_{0}U(z)}{\left|\nabla U(z)\right|}m(z)\ell(dz),\nonumber\\
\widehat{A}(x,i)&=\oint_{d_{i}(x)}2\beta |\nabla U(z)|^{2}m_{x,i}(z)\ell(dz)=\frac{1}{T_{i}(x)}\oint_{d_{i}(x)}\frac{2\beta \nabla U(z)\cdot\nabla U(z)}{\left|\nabla U(z)\right|}m(z)\ell(dz).
\end{align*}
and then consider the one-dimensional process $Y_{t}$ on the branch
$I_{i}$ governed by the infinitesimal generator
\begin{equation}
\mathcal{L}^{Y}_{i}g(x)= \widehat{\mathcal{L}_{0}U}(x,i) g'(x)+\frac{1}{2}\widehat{A}(x,i)g''(x).\label{Eq:AveragedOperator}
\end{equation}

Within each edge $I_{i}$ of $\Gamma$,
$Q(Z^{\veps}_{t}) = (U(Z^{\veps}_{t}), i(Z^{\veps}_{t}))$ converges
as $\veps\downarrow 0$ to a process with infinitesimal generator
$\mathcal{L}^{Y}_{i}$. In order to uniquely define the limiting
process, we need to specify the behavior at the vertexes of the tree,
which amounts to imposing restrictions on the domain of definition of
the generator, denoted by $\mathcal{L}^{Y}$, of the Markov
process. For this purpose, we have the following definition
\begin{definition}\label{Def:ProcessOnTree}
We say that $g$ belongs in the domain of definition of $\mathcal{L}^{Y}$, denoted by $\mathcal{D}(\mathcal{L}^{Y})$, of the diffusion $Y$, if
\begin{enumerate}
\item{The function $g(x)$ is twice continuously differentiable in the interior of an edge $I_{i}$.}
\item{The function $x\mapsto\mathcal{L}^{Y}_{i}g(x)$ is continuous on $\Gamma$.}
\item{At each interior vertex $O_{j}$ with edges $I_{i}$ that meet at $O_{j}$, the following gluing condition holds
    \begin{equation*}
    \sum_{i: I_{i}\sim O_{j}}\pm b_{ji}D_{i}g(O_{j})=0
    \end{equation*}
    where, if $\gamma_{ji}$ is the separatrices  curves that meet at $O_{j}$,  we have set
    \begin{equation*}
    b_{ji}=\oint_{\gamma_{ji}}\frac{2\beta \left|\nabla U(z)\right|^{2}}{\left|\nabla U(z)\right|}m(z)\ell(dz)=2\beta\oint_{\gamma_{ji}}\left|\nabla U(z)\right|m(z)\ell(dz).
    \end{equation*}
    Here one chooses $+$ or $-$ depending on whether the value of $U$ increases or decreases respectively along
 the edge $I_{i}$ as we approach $O_{j}$.  $D_{i}$ represents the derivative in the direction of the edge $I_{i}$.
    }
\end{enumerate}
Moreover, within each edge $I_{i}$ the  process $Y_{t}$ is a diffusion process with infinitesimal generator $\mathcal{L}^{Y}_{i}$.
\end{definition}

Consider now the process $Y_{t}$ that has the aforementioned
$\mathcal{L}^{Y}$ as its infinitesimal generator with domain of
definition $\mathcal{D}(\mathcal{L}^{Y})$, as defined in Definition
\ref{Def:ProcessOnTree}.  Such a process is a continuous strong Markov
process, e.g., \cite[Chapter 8]{FW2}. Then, for any $T>0$,
$Q(Z^{\veps}_{t})$ converges weakly in $\mathcal{C}([0,T];\Gamma)$ to
the Markov process $Y_{t}$ on the tree as $\veps\downarrow 0$.  In
particular, we have the following theorem.
\begin{theorem}[Theorem 2.1 of \cite{FreidlinWeber2004}]\label{T:LimitOnGraphb}
Let $Z^{\veps}_{t}$ be the process that satisfies (\ref{Eq:RegularizedIrreversibleLangevin}). Assume Conditions \ref{A:Assumption2}, \ref{A:Assumption3} and \ref{A:Assumption4}. Let $T>0$ and consider
 the Markov process on the tree $\left\{Y_{t},t\in[0,T]\right\}$ as defined by Definition \ref{Def:ProcessOnTree}. We have
\begin{equation}
 Q(Z^{\veps}_{\cdot})\rightarrow Y_{\cdot}, \text{ weakly in } \mathcal{C}([0,T];\Gamma), \text{ as }\veps\downarrow 0.\nonumber
\end{equation}
\end{theorem}

\section{Analysis of the numerical HMM method}\label{sec:analysis}

To analyze the numerical scheme, following \cite{TaoOwhadiMarsden2010}, let us
assume that there exists a random variable $\Phi_{h}^{\alpha}(z)$ and
an $h_{0}$ such that for all $0< h\leq h_{0}$ and $\alpha = 0$ or
$\alpha =\frac{1}{\veps}$ one has the estimate
\begin{align}
\left(\mathbb{E}\left|\Phi_{h}^{\alpha}(z)-z+h \nabla U(z)-\alpha h \wt{C}(z)-\sqrt{h}\sqrt{2\beta}\xi-\sqrt{h}\sqrt{\kappa\alpha}\sigma(z)\xi'\right|^{2}\right)^{1/2}&\leq C h^{3/2}(1+\alpha)^{3/2},\label{Eq:Integrator1}
\end{align}
where $\xi,\xi'$ are independent standard normal random variable. Then, the algorithm becomes
\begin{align}
\bar{Z}^{\veps}_{0}&=z_{0}\nonumber\\
\bar{Z}^{\veps}_{(\kappa+1)\delta}&=\left(\Phi_{\delta-\tau}^{0}\circ \Phi_{\tau}^{\frac{1}{\veps}}\right)(\bar{Z}^{\veps}_{\kappa\delta})\label{Eq:HMMscheme}
\end{align}

Recall from Theorem~\ref{T:LimitOnGraphb} that it is important to keep in mind that $Z^{\veps}_{t}$ does not converge to somewhere when $\veps\downarrow 0$. What converges to somewhere, i.e., to the diffusion on the tree, is $Q(Z^{\veps})=(U(Z^\veps),i(Z^{\veps}))$.
Then, we have the following theorem.


\begin{theorem}\label{T:ConvNumericalScheme}
Assume the conditions of Theorem \ref{T:LimitOnGraphb} and that $\veps,\delta,\tau\downarrow 0$ are such that $\frac{\delta\veps}{\tau}, \frac{\tau}{\veps},\left(\frac{\tau}{\veps}\right)^{3/2}\frac{1}{\delta}\downarrow 0$. Then, for $\tau<\delta<\frac{\tau}{\veps}\ll 1$ sufficiently small,  the process
$Q(\bar{Z}^{\veps}_{n\delta})=(U(\bar{Z}^{\veps}_{n\delta}),i(\bar{Z}^{\veps}_{n\delta}))$ (where $\bar{Z}^{\veps}$ is the process from (\ref{Eq:HMMscheme})) converges in distribution  to the process $Y_{\cdot}$ as defined in Definition \ref{Def:ProcessOnTree}. In addition, convergence to the invariant
measure $\mu$ of the $Y$ process holds, in the sense that for any bounded and uniformly Lipschitz test function $\phi$  we have that for all $t>0$
\begin{align}
  \lim_{h\downarrow 0}\lim_{\veps,\delta, \frac{\delta\veps}{\tau}, \frac{\tau}{\veps},\left(\frac{\tau}{\veps}\right)^{3/2}\frac{1}{\delta}\downarrow 0} \frac{1}{h}\int_{t}^{t+h}E_{\pi} \phi(\bar{Z}^{\veps}_{s})ds&= E_{\mu} \widehat{\phi}(Y_{t})\,, \nonumber
\end{align}
where $\pi$ is the invariant measure of the continuous process $Z^{\veps}$.
\end{theorem}

The proof of this theorem is done in two steps. In the first step, in Section \ref{S:OneWellCase}, we consider the case of a single well. Then, in the second step in Section \ref{S:MultiWellCase}, we complete the proof by considering the general multiple well case.

\subsection{The case of one well}\label{S:OneWellCase}
Let us assume that there is only one well, i.e., that for any $T>0$ and $s\in[0,T]$, we have $i(\bar{Z}^{\veps}_{s})=1$. In this case, we simply have $Q(\bar{Z}^{\veps}_{\cdot})=(U(\bar{Z}^{\veps}_{\cdot}),1)$ and we are interested in the asymptotic behavior of the process $Q_{s}=U(\bar{Z}^{\veps}_{s})$. Going back to (\ref{Eq:Integrator1}) we have the following lemma.
\begin{lemma}\label{L:U_intergator2}
  Consider $z$ such that $|\nabla U(z)|\leq C<\infty$. Let us define
  $\Psi_{h}^{\alpha}(z)=U(\Phi_{h}^{\alpha}(z))$. Then, there exists
  $h_{0}<\infty$ such that for all $0< h\leq h_{0}$ and $\alpha = 0$
  or $\alpha=\frac{1}{\veps}$, one has the estimate
\begin{multline}
  \left(\mathbb{E}\left|\Psi_{h}^{\alpha}(z)-U(z)-h\left(- |\nabla U(z)|^{2}+\beta \text{tr}\left[D^{2}U(z)\right]\right)-\sqrt{h}\sqrt{2\beta}\nabla U(z)\cdot \xi\right|^{2}\right)^{1/2} \leq \\
  \leq C h^{3/2}(1+\alpha)^{3/2}\,,\nonumber
\end{multline}
where $\xi$ is a standard multidimensional normal random variable.
\end{lemma}
\begin{proof}
By applying Taylor expansion to $U(\Phi_{h}^{\alpha}(z))$ up to second order with respect to $0<h\ll h_{0}$ and using (\ref{Eq:Integrator1}) we get for $h$ sufficiently small
\begin{align}
U(\Phi_{h}^{\alpha}(z))&=U(z)+\nabla U(z)\left(\Phi_{h}^{\alpha}(z)-z\right)+\frac{1}{2}\left(\Phi_{h}^{\alpha}(z)-z\right)^{T}D^{2}U(z)\left(\Phi_{h}^{\alpha}(z)-z\right)+ o((\Phi_{h}^{\alpha}(z)-z)^{2}) \nonumber\\
&=U(z)+ \nabla U(z)\left(h [-\nabla U(z)+\alpha  \wt{C}(z)]+\sqrt{h}\sqrt{2\beta}\xi+\sqrt{h}\sqrt{\kappa\alpha}\sigma(z)\xi' +I(\alpha,h)\right)\nonumber\\
&\quad+ \frac{1}{2}\left(h [-\nabla U(z)+\alpha  \wt{C}(z)]+\sqrt{h}\sqrt{2\beta}\xi+\sqrt{h}\sqrt{\kappa\alpha}\sigma(z)\xi'+I(\alpha,h)\right)^{T}D^{2}U(z)\times\nonumber\\
&\quad\qquad\times\left(h [-\nabla U(z)+\alpha  \wt{C}(z)]+\sqrt{h}\sqrt{2\beta}\xi+\sqrt{h}\sqrt{\kappa\alpha}\sigma(z)\xi'+I(\alpha,h)\right)+ o((\Phi_{h}^{\alpha}(z)-z)^{2}),\nonumber
\end{align}
where $\left(\EE I^{2}(\alpha,h)\right)^{1/2}\leq C h^{3/2}(1+\alpha)^{3/2}$. Using now the assumptions from Conditions~\ref{A:Assumption3} and \ref{A:Assumption4} that $\nabla U(z)C(z)=0$, $\sigma(z)\sigma^{T}(z)\nabla U(z)=0$ and expanding the quadratic term, the previous expression simplifies to
\begin{align*}
U(\Phi_{h}^{\alpha}(z))
&=U(z)+ \left(h \left[-|\nabla U(z)|^{2}+\beta \text{tr}\left[D^{2}U(z)\right]\right]\right)+\sqrt{h}\sqrt{2\beta}\nabla U(z)\xi+ \nabla U(z)I(\alpha,h)+R_{1}(\alpha,h),
\end{align*}
where $\left(\EE R^{2}_{1}(\alpha,h)\right)^{1/2}=o( h^{3/2}(1+\alpha)^{3/2})$. The latter, essentially concludes the proof of the lemma.
\end{proof}

For notational convenience, let us define the operator on test functions $f\in\mathcal{C}^{2}(\mathbb{R})$,
\begin{align}
\mathcal{L}_{Q}f(z)&=\left[-|\nabla U(z)|^{2}+\beta \text{tr}\left[D^{2}U(z)\right]\right]f^{'}(U(z))+\beta |\nabla U(z)|^{2} f^{''}(U(z))\label{Eq:LqOperator}
\end{align}

Next we have the following lemma for the numerical approximation HMM scheme (\ref{Eq:HMMscheme}).
\begin{lemma}\label{L:TimeStepIntegrator}
Let $f\in\mathcal{C}^{2}(\mathbb{R})$. Then for $\bar{Z}^{\veps}_{t}$ given by (\ref{Eq:HMMscheme}) we have, as $\tau<\delta<\tau/\veps$, $\delta,\frac{\tau}{\veps}\downarrow 0$
\begin{align}
&\EE\left(f(U(\bar{Z}^{\veps}_{(n+1)\delta}))-f(U(\bar{Z}^{\veps}_{n\delta}))\right)=\delta \EE\mathcal{L}_{Q}f(\bar{Z}^{\veps}_{n\delta})+O\left(\delta^{3/2}+\left(\frac{\tau}{\veps}\right)^{3/2}\right)\,.\nonumber
\end{align}
\end{lemma}
\begin{proof}
We start by noticing that Lemma \ref{L:U_intergator2} implies that
\begin{align}
U(\bar{Z}^{\veps}_{n\delta+\tau})&=U(\bar{Z}^{\veps}_{n\delta})+\tau\left(- |\nabla U(\bar{Z}^{\veps}_{n\delta})|^{2}+\beta \text{tr}\left[D^{2}U(\bar{Z}^{\veps}_{n\delta})\right]\right)+\sqrt{\tau}\sqrt{2\beta}\nabla U(\bar{Z}^{\veps}_{n\delta})\xi_{n}+R_{2,n},\nonumber
\end{align}
where $\left(\EE R_{2,n}^{2}\right)^{1/2}\leq C \left(\frac{\tau}{\veps}\right)^{3/2}$. The last display and smoothness of the test function $f$, implies that
\begin{align}
\EE f(U(\bar{Z}^{\veps}_{n\delta+\tau}))&=\EE f(U(\bar{Z}^{\veps}_{n\delta}))+\tau \EE\mathcal{L}_{Q}f(\bar{Z}^{\veps}_{n\delta})+\EE R_{2,n},\nonumber
\end{align}
where with some abuse of notation we still denote $R_{2,n}$ the error
term which again satisfies
$\left(\EE R_{2,n}^{2}\right)^{1/2}\leq C
\left(\frac{\tau}{\veps}\right)^{3/2}$.

In a similar manner, we also obtain that
\begin{align}
\EE f(U(\bar{Z}^{\veps}_{(n+1)\delta}))&=\EE f(U(\bar{Z}^{\veps}_{n\delta+\tau}))+(\delta-\tau) \EE\mathcal{L}_{Q}f(\bar{Z}^{\veps}_{n\delta+\tau})+ \EE R_{3,n},\nonumber
\end{align}
where $\left(\EE R^{2}_{3,n}\right)^{1/2}\leq C (\delta-\tau)^{3/2}$.

Hence, we get
\begin{align}
\EE\left(f(U(\bar{Z}^{\veps}_{(n+1)\delta}))-f(U(\bar{Z}^{\veps}_{n\delta}))\right)&=\delta \EE\mathcal{L}_{Q}f(\bar{Z}^{\veps}_{n\delta})\nonumber\\
&+(\delta-\tau)\EE\left(\mathcal{L}_{Q}f(\bar{Z}^{\veps}_{n\delta})-\mathcal{L}_{Q}f(\bar{Z}^{\veps}_{n\delta+\tau})\right)+\EE R_{2,n}+\EE R_{3,n}.\nonumber
\end{align}

We further notice that by the regularity of $U$ and $f$ we have
\begin{equation*}
\begin{aligned}
\left(\EE\left(\mathcal{L}_{Q}f(\bar{Z}^{\veps}_{n\delta})-\mathcal{L}_{Q}f(\bar{Z}^{\veps}_{n\delta+\tau})\right)^{2}\right)^{1/2}\leq C \left(\EE\left|\bar{Z}^{\veps}_{n\delta+\tau}-\bar{Z}^{\veps}_{n\delta}\right|^{2}\right)^{1/2}&\leq C(\sqrt{\tau}+\tau/\veps)\\
  &\leq C(\sqrt{\delta}+\tau/\veps).
\end{aligned}
\end{equation*}

Putting the estimates together we obtain the statement of the lemma.
\end{proof}

Let us recall now the operator $\mathcal{L}_{Q}f(z)$ defined by (\ref{Eq:LqOperator}) and let us recall the ``averaged'' generator $\mc{L}^Y=\mc{L}^Y_1$ defined via (\ref{Eq:AveragedOperator}) (recall that the single edge case is considered at the moment). We want to prove that the process $U(\bar{Z}^{\veps}_{n\delta})$ converges in distribution to the process with generator $\mc{L}^{Y}$.

For this purpose, we may use  Theorem 1 of  \cite[Chapter 2]{Skorokhod}. By Lemma \ref{L:TimeStepIntegrator} we have
\begin{equation}
\begin{aligned}
&\EE\left[f(U(\bar{Z}^{\veps}_{n\delta}))-f(U(z))-\int_{0}^{n\delta}\mathcal{L}^{Y}f(U(\bar{Z}^{\veps}_{s}))ds\right]=\\
&\quad =\EE\sum_{k=0}^{n-1}\left[f(U(\bar{Z}^{\veps}_{(k+1)\delta}))-f(U(\bar{Z}^{\veps}_{k\delta}))-\int_{k\delta}^{(k+1)\delta}\mathcal{L}^{Y}f(U(\bar{Z}^{\veps}_{s}))ds\right]\\
&\quad =\delta \sum_{k=0}^{n-1} \EE\left[\mathcal{L}_{Q}f(\bar{Z}^{\veps}_{k\delta})-\mathcal{L}^{Y}f(U(\bar{Z}^{\veps}_{k\delta}))\right]+ n\delta \EE I_{0}\\
& \quad = \EE\int_{0}^{n\delta}\left[\mathcal{L}_{Q}f(\bar{Z}^{\veps}_{s})-\mathcal{L}^{Y}f(U(\bar{Z}^{\veps}_{s}))\right]ds+ n\delta \EE I_{0}, \label{Eq:MartingaleProblem_ApproximatingSequence}
\end{aligned}
\end{equation}
where $\left(\EE I_{0}^{2}\right)^{1/2}\leq C\left(\delta^{3/2}+\left(\frac{\tau}{\veps}\right)^{3/2}\right)$.

Notice that $\mathcal{L}^{Y}f(U(z))=\widehat{\mathcal{L}_{Q}f}(U(z))$. Essentially, for a nice function $g = \mc{L}_Q f$, we need tight estimates for $\EE\int_{0}^{n\delta}\left[g(\bar{Z}^{\veps}_{s})-\widehat{g}(U(\bar{Z}^{\veps}_{s}))\right]ds$, where we recall that $\bar{Z}^{\veps}_{s}$ is the approximating process and $\widehat{g}$ is the average on graph as defined in \eqref{Eq:AverageOnGraph}. We can write
\begin{equation}
\begin{aligned}
\EE\int_{0}^{n\delta}\left[\mathcal{L}_{Q}f(\bar{Z}^{\veps}_{s})-\mathcal{L}^{Y}f(U(\bar{Z}^{\veps}_{s}))\right]ds&=n \delta\left[\frac{1}{n}\sum_{k=0}^{n-1}\left(\EE\mathcal{L}_{Q}f(\bar{Z}^{\veps}_{k\delta})-\frac{1}{\tau}\int_{k\tau}^{(k+1)\tau}\EE\mathcal{L}_{Q}f(Z^{\veps}_{s})ds\right)\right]\nonumber\\
&+n \delta\left[\frac{1}{n\tau}\int_{0}^{n\tau}\EE\mathcal{L}_{Q}f(Z^{\veps}_{s})ds-\frac{1}{n\tau}\int_{0}^{n\tau}\EE\mathcal{L}^{Y}f(U(Z^{\veps}_{s}))ds\right]\nonumber\\
&+n \delta\left[\frac{1}{n}\sum_{k=0}^{n-1}\left(\int_{k\tau}^{(k+1)\tau}\EE\mathcal{L}^{Y}f(U(Z^{\veps}_{s}))ds-\EE\mathcal{L}^{Y}f(U(\bar{Z}^{\veps}_{k\delta}))\right)\right]\nonumber\\
&=n\delta\left[J_{1}^{n}+J_{2}^{n}+J_{3}^{n}\right].\label{Eq:MartingaleProblem_ApproximatingSequence2}
\end{aligned}
\end{equation}

By the estimate (A.103)-(A.104) of \cite{TaoOwhadiMarsden2010} we have that for an unimportant constant $C<\infty$
\begin{align}
|J_{1}^{n}+J_{3}^{n}|&\leq C \left(\sqrt{\frac{\tau}{\veps}}+\sqrt{n\delta}+n\delta+n\left(\frac{\tau}{\veps}\right)^{3/2}\right)e^{C\frac{n\tau}{\veps}}.\label{Eq:J1J3bounds}
\end{align}

It remains to treat the term $J_{2}^{n}=\frac{1}{n\tau}\int_{0}^{n\tau}\EE\mathcal{L}_{Q}f(Z^{\veps}_{s})ds-\frac{1}{n\tau}\int_{0}^{n\tau}\EE\mathcal{L}^{Y}f(U(Z^{\veps}_{s}))ds$. Standard PDE arguments, e.g., \cite[Section 3.2]{FreidlinWeber2004}, show that for any point $(z,i)\in I_{1}$ that is not a vertex, and for $g\in\mathcal{C}^{2+\alpha}$, the PDE
\begin{equation}
-\widehat{\mathcal{L}}u(z)=g(z)-\widehat{g}(U(z)), \quad\text{for }z\in d_{1}(x),\label{Eq:AuxilliaryPDE}
\end{equation}
has a unique (up to constants) $\mathcal{C}^{2+\alpha'}$ solution with
$\alpha'\in(0,\alpha)$. We fix the free constant by setting
$\widehat{u}(x,1)=0$. Then, the solution $u(z)$ can be written as
\[
u(z)=\int_{0}^{\infty}\mathbb{E}_{z}\left[g(\widehat{Z}_{s})-\widehat{g}(U(\widehat{Z}_{s}))\right]ds.
\]

Moreover, there exist a constant $\lambda=\lambda(z,1)>0$ such that for $z\in d_{1}(x)$,
\begin{equation*}
|u(z)|\leq \frac{2}{\lambda}\sup_{z\in d_{1}(x)}\bigl\lvert g(z)-\widehat{g}(U(z))\bigr\rvert.\label{Eq:BoundAuxilliaryPDE}
\end{equation*}

Notice that in the case that we can take $\kappa=0$, \textit{i.e.}, when the dynamical system $\dot{z}_{t}=C(z_{t})$ has a unique invariant measure on the connected component $d_{1}(x)$, then we simply have $\widehat{\mathcal{L}}u(z)=C(z)\nabla u(z)$. If we cannot take $\kappa=0$, then $\widehat{\mathcal{L}}u(z)$ is given by (\ref{Eq:FastOperator}).
Applying It\^{o}'s formula to $u$ we obtain that
\begin{align}
u(Z^{\veps}_{n\tau})&=u(z)+\frac{1}{\veps}\int_{0}^{n\tau}\widehat{\mathcal{L}} u(Z^{\veps}_{s})ds+\int_{0}^{n\tau}\mathcal{L}_{0} u(Z^{\veps}_{s})ds\nonumber\\
&+\sqrt{\frac{\kappa}{\veps}}\int_{0}^{n\delta}\nabla u(Z^{\veps}_{s})\sigma(Z^{\veps}_{s})dW^{o}_{s}+\sqrt{2\beta}\int_{0}^{n\delta}\nabla u(Z^{\veps}_{s})dW_{s}.\nonumber
\end{align}

Recalling now that $u$ solves (\ref{Eq:AuxilliaryPDE}), we obtain by rearranging the last display and taking expected value
\begin{align}
 J_{2}^{n}&=\frac{1}{n\tau}\int_{0}^{n\tau}\EE\mathcal{L}_{Q}f(Z^{\veps}_{s})ds-\frac{1}{n\tau}\int_{0}^{n\tau}\EE\mathcal{L}^{Y}f(U(Z^{\veps}_{s}))ds\nonumber\\
&=\frac{\veps}{n\tau}\left[\EE\left(u(Z^{\veps}_{n\tau})-u(z)\right)+\EE\int_{0}^{n\tau}\mathcal{L}_{0} u(Z^{\veps}_{s})ds\right],\nonumber
\end{align}
which then, due to the boundedness of $u$ and its derivatives, gives
\begin{align}
 |J_{2}^{n}|&\leq C \left(\frac{\veps}{n\tau}+\veps\right).\label{Eq:J2bound}
\end{align}

Thus, using estimates (\ref{Eq:J1J3bounds})-(\ref{Eq:J2bound}), (\ref{Eq:MartingaleProblem_ApproximatingSequence}) gives
\begin{align}
&\left|\EE\left[f(U(\bar{Z}^{\veps}_{n\delta}))-f(U(z))-\int_{0}^{n\delta}\mathcal{L}^{Y}f(U(\bar{Z}^{\veps}_{s}))ds\right]\right|\leq \nonumber\\
&\qquad \leq C n\delta\left[\left(\sqrt{\frac{\tau}{\veps}}+\sqrt{n\delta}+n\delta+n\left(\frac{\tau}{\veps}\right)^{3/2}\right)e^{C\frac{n\tau}{\veps}}+\frac{\veps}{n\tau}+\veps+\delta^{3/2}+\left(\frac{\tau}{\veps}\right)^{3/2}\right].\label{Eq:MartingaleProblem_ApproximatingSequence3}
\end{align}

Choosing now $n$ such that $\sqrt{\frac{n\tau}{\veps}}e^{C\frac{n\tau}{\veps}}\sim\left(\frac{\tau}{\veps\delta}\right)^{1/4}$, and recalling the requirement $\tau<\delta<\frac{\tau}{\veps}\ll 1$ we obtain from (\ref{Eq:MartingaleProblem_ApproximatingSequence3})
\begin{align}
&\frac{1}{n\delta}\left|\EE\left[f(U(\bar{Z}^{\veps}_{n\delta}))-f(U(z))-\int_{0}^{n\delta}\mathcal{L}^{Y}f(U(\bar{Z}^{\veps}_{s}))ds\right]\right|\leq \nonumber\\
&\qquad\leq C \left[\left(\frac{\delta\veps}{\tau}\right)^{1/4}+\left(\frac{\delta\veps}{\tau}\right)^{1/2}+\left(\frac{\tau}{\veps}\right)^{3/2}\frac{1}{\delta}\sqrt{\frac{\delta\veps}{\tau}}+\frac{1}{\log\left(\frac{\tau}{\veps\delta}\right)}+\frac{\tau}{\delta}+\delta^{3/2}+\left(\frac{\tau}{\veps}\right)^{3/2}\right]\nonumber\\
&\qquad \rightarrow 0, \text{ as }\frac{\delta\veps}{\tau}, \left(\frac{\tau}{\veps}\right)^{3/2}\frac{1}{\delta}\downarrow 0.\nonumber
\end{align}

Hence, by  Theorem 1 of  \cite[Chapter 2]{Skorokhod}, we have obtained that $U(\bar{Z}^{\veps}_{n\delta})$ converges in distribution to the process $Y$ on the graph (for the moment with just one edge) with generator $\mathcal{L}^{Y}$.

Let us next discuss convergence to the invariant measure. Since the
invariant measure for the original process $Z^{\veps}$ is the Gibbs
measure $\pi$, we get that the invariant measure for the process Y on
the tree $\Gamma$ is nothing else but the projection of $\pi$ on
$\Gamma$, say $\mu$. In particular for any Borel set $\gamma\subset
\Gamma$, we have $\mu(\gamma)=\pi(\Gamma^{-1}(\gamma))$.

Then, from the weak convergence of $Q(\bar{Z}^{\veps}_{n\delta})$ to the process $Y$ and the uniform mixing properties of $Z^{\veps}_{t}$ and $Y_{t}$, we get that for any bounded and uniformly Lipschitz test function $\phi$  that for all $t>0$
\begin{align}
\lim_{h\downarrow 0}\lim_{\veps,\delta, \frac{\delta\veps}{\tau}, \frac{\tau}{\veps},\left(\frac{\tau}{\veps}\right)^{3/2}\frac{1}{\delta}\downarrow 0} \frac{1}{h}\int_{t}^{t+h}E_{\pi} \phi(\bar{Z}^{\veps}_{s})ds&=
\lim_{h\downarrow 0}\lim_{\veps,\delta, \frac{\delta\veps}{\tau}, \frac{\tau}{\veps},\left(\frac{\tau}{\veps}\right)^{3/2}\frac{1}{\delta}\downarrow 0} \frac{1}{h}\int_{t}^{t+h}E_{\pi} \widehat{\phi}(Q(\bar{Z}^{\veps}_{s}))ds\nonumber\\
 &=E_{\mu} \widehat{\phi}(Y_{t}). \label{Eq:F_convergence1}
\end{align}
The latter establishes  Theorem \ref{T:ConvNumericalScheme} in the one well case.
\subsection{The multi-well case}\label{S:MultiWellCase}

The goal of this section is to establish that Theorem \ref{T:ConvNumericalScheme} holds in the general multi-well case, i.e., when $m>1$. First we need to define certain objects. Let us consider $\theta>0$ small and for an edge $I_{i}$ of the graph set
\[
 I^{\theta}_{i}=\{(x,i)\in I_{i}: \text{dist}((x,i),\partial I_{i})>\theta\}
\]
and define
\[
 \bar{\tau}_{i}=\inf\{t>0: Q(\bar{Z}^{\veps}_{t})\notin I^{\theta}_{i}\} \,.
\]
Thus, $I^{\theta}_i$ is the interior part of the edge $I_i$ and
$\bar{\tau}_i$ is the first exit time of the interior part.

In addition, for $\zeta>0$ and for a vertex of the graph $O_{j}$ and a segment $I_i \sim O_j$, let us define the following quantities as in \cite[Chapter 8]{FreidlinWentzell88}:
\begin{align}
  D_i &= \left\{z \in E: Q(z) \in I_i^{\circ} \right\} \nonumber \\
  D_{i}(U_{1},U_{2})&=\left\{z\in D_{i}: U_{1}<U(z)<U_{2}\right\}\nonumber\\
  D_{j}(\pm\zeta)&=\left\{z\in E: U(O_{j})-\zeta<U(z)<U(O_{j})+\zeta\right\}\,\nonumber\\
  D(\pm\zeta)&=\bigcup_{j}D_{j}(\pm\zeta) \nonumber\\
  C_{j}&=\left\{z\in E: Q(z)=O_{j}\right\}\nonumber \\
  C_{ji}(\zeta)&=\left\{z\in D_{i}: U(z)=U(O_{j})\pm\zeta\right\}\,\nonumber\\
  C_{ji}&=C_{j}\bigcap \partial D_{i}\nonumber\\
  C_{i}(U)&=\left\{z\in \bar{D}_{i}: U(z)=U\right\}.\nonumber
\end{align}

Here $I_i^{\circ}$ denotes the open interior of $I_i$.  Let us then
also define the first exit time of the process $\bar{Z}^{\veps}_{t}$
from $D_{j}(\pm\zeta)$ as follows
\[
 \bar{\tau}^{\veps}_{j}(\pm\zeta)=\inf\{t>0: \bar{Z}^{\veps}_{t}\notin D_{j}(\pm\zeta)\} \,.
\]
Following the proof of \cite[Theorem 8.2.2]{FreidlinWentzell88}, see
also \cite{FW2}, the statement of Theorem \ref{T:ConvNumericalScheme}
will follow if we show that in the limit as
$\veps,\delta,\tau\downarrow 0$ the process $\bar{Z}^{\veps}_{\cdot}$
behaves within a given $i$ well according to the generator
$\mathcal{L}^{Y}_{i}$, it spends zero time in exterior and interior
vertices and that the probabilistic behavior at the vertices leads to
the gluing condition of Definition \ref{Def:ProcessOnTree}. To be
precise, following the proof of \cite[Theorem
8.2.2]{FreidlinWentzell88}, Theorem \ref{T:ConvNumericalScheme}
follows if we prove Lemmas \ref{L:ConvergenceWithinEdge},
\ref{L:ExteriorVertex}, \ref{L:InteriorVertex} and
\ref{L:ProbabilitiesAtVertex} below.

\begin{lemma}\label{L:ConvergenceWithinEdge}
Let $f\in\mathcal{C}^{2}_{b}(\mathbb{R})$ and $\theta>0$ such that $I^{\theta}_{i}\neq \emptyset$ for all $i\in\{1,\cdots, m\}$. Assume the conditions of Theorem~\ref{T:ConvNumericalScheme}.  Then,  uniformly in $z\in D_{i}^{\theta}=\{z\in E: Q(z)\subset I^{\theta}_{i}\}$, we have that
\begin{multline}
  \left\lvert \EE\left[f(U(\bar{Z}^{\veps}_{n\delta \wedge \bar{\tau}_{i}}))-f(U(z))-\int_{0}^{n\delta\wedge \bar{\tau}_{i}}\mathcal{L}^{Y}f(U(\bar{Z}^{\veps}_{s}))ds\right]\right\rvert\\
  \leq C n\delta \left( \left(\sqrt{\frac{\tau}{\veps}}+\sqrt{n\delta}+n\delta+n\left(\frac{\tau}{\veps}\right)^{3/2}\right)e^{C\frac{n\tau}{\veps}}+\frac{\veps}{n\tau}+\frac{\tau}{\delta}+\delta^{3/2}+\left(\frac{\tau}{\veps}\right)^{3/2}\right),
\label{Eq:MartingaleProblem_ApproximatingSequenceGen2}
\end{multline}
for some constant $C<\infty$. In particular, choosing  $n$ such that $\sqrt{\frac{n\tau}{\veps}}e^{C\frac{n\tau}{\veps}}\sim\left(\frac{\tau}{\veps\delta}\right)^{1/4}$, we obtain that
\begin{align}
&\frac{1}{n\delta}\left\lvert \EE\left[f(U(\bar{Z}^{\veps}_{n\delta \wedge \bar{\tau}_{i}}))-f(U(z))-\int_{0}^{n\delta\wedge \bar{\tau}_{i}}\mathcal{L}^{Y}f(U(\bar{Z}^{\veps}_{s}))ds\right]\right\rvert\leq \nonumber\\
&\qquad\leq C \left[\left(\frac{\delta\veps}{\tau}\right)^{1/4}+\left(\frac{\delta\veps}{\tau}\right)^{1/2}+\left(\frac{\tau}{\veps}\right)^{3/2}\frac{1}{\delta}\sqrt{\frac{\delta\veps}{\tau}}+\frac{1}{\log\left(\frac{\tau}{\veps\delta}\right)}+\frac{\tau}{\delta}+\delta^{3/2}+\left(\frac{\tau}{\veps}\right)^{3/2}\right]\nonumber\\
&\qquad \rightarrow 0, \text{ as }\frac{\delta\veps}{\tau}, \left(\frac{\tau}{\veps}\right)^{3/2}\frac{1}{\delta}\downarrow 0.\nonumber
\end{align}

\end{lemma}

Lemma \ref{L:ConvergenceWithinEdge} follows directly by the arguments of Section \ref{S:OneWellCase}.  In particular Lemma \ref{L:ConvergenceWithinEdge} implies that if $\delta, \frac{\delta\veps}{\tau}, \frac{\tau}{\veps}, \left(\frac{\tau}{\veps}\right)^{3/2}\frac{1}{\delta}\downarrow 0$, then  the process
$Q(\bar{Z}^{\veps}_{n\delta \wedge \bar{\tau}_{i}})$ converges in distribution, within edge $I_{i}$, to the process with generator $\mathcal{L}^{Y}_{i}$ as defined by (\ref{Eq:AveragedOperator}).

\begin{lemma}\label{L:ExteriorVertex}
Let $O_{j}$ be an exterior vertex of the graph $\Gamma$.  Assume the conditions of Theorem \ref{T:ConvNumericalScheme}.  Then,  there exists $\zeta_{0}> 0$, such that for all $0<\zeta\leq \zeta_{0}$ and for all $z\in D_{j}(\pm\zeta)$, there exists a constant $C<\infty$ such that
\begin{align*}
\EE_{z}\bar{\tau}^{\veps}_{j}(\pm\zeta)&\leq C(\zeta+\delta+(\tau/\veps)^{3/2}).
\end{align*}
In other words, for every $\eta>0$ and for $0<\max\{\delta,(\tau/\veps)^{3/2}\}<\zeta\leq \zeta_{0}$ sufficiently small, we have that
\begin{align*}
\EE_{z}\bar{\tau}^{\veps}_{j}(\pm\zeta)&\leq \eta .
\end{align*}

\end{lemma}

\begin{lemma}\label{L:InteriorVertex}
Let $O_{j}$ be an interior vertex of the graph $\Gamma$.
Assume the conditions of Theorem \ref{T:ConvNumericalScheme}. Then, there exists $\zeta_{0}>\veps^{\alpha}$ for some exponent $\alpha>0$, such that for all $0<\veps^{\alpha}<\zeta\leq \zeta_{0}$ and for all $z\in D_{j}(\pm\zeta)$
\begin{align*}
\EE_{z}\bar{\tau}^{\veps}_{j}(\pm\zeta)&\leq C \zeta^{2}|\ln\zeta|.
\end{align*}
In other words for every $\eta>0$, there exists $\zeta_{0}>\veps^{\alpha}$ for some exponent $\alpha>0$, such that for all $0<\veps^{\alpha}<\zeta\leq \zeta_{0}$ and for all $z\in D_{j}(\pm\zeta)$
\begin{align*}
\EE_{z}\bar{\tau}^{\veps}_{j}(\pm\zeta)&\leq\eta \zeta.
\end{align*}
\end{lemma}

\begin{lemma}\label{L:ProbabilitiesAtVertex}
  For $I_{i}\sim O_{j}$ define $b_{ji}$ as in Definition
  \ref{Def:ProcessOnTree} and set
  $p_{ji}=\frac{b_{ji}}{\sum_{i: I_{i}\sim O_{j}} b_{ji}}$.
  Assume the conditions of Theorem \ref{T:ConvNumericalScheme}.  Then, for every $\eta>0$, there
  exists $\zeta_{0}> \max\{\delta,(\tau/\veps)^{3/2}\}$, such that for
  all $0<\max\{\delta,(\tau/\veps)^{3/2}\}<\zeta\leq \zeta_{0}$ there
  exists $\zeta'_{0}=\zeta'_{0}(\zeta)$ such that for all sufficiently
  small $\veps,\delta,\tau$
\begin{align}
  \left\lvert\PP_{z}\left(\bar{Z}^{\veps}_{\bar{\tau}^{\veps}_{j}(\pm\zeta)}\notin \overline{\partial D_{j}(\pm\zeta)}\cap  I^{\circ}_{i}\right)-p_{ji}\right\rvert<\eta.\nonumber
\end{align}
for all $z\in \overline{D_{j}(\pm \zeta'_{0})}$.
\end{lemma}

Lemmas \ref{L:ExteriorVertex}, \ref{L:InteriorVertex} and \ref{L:ProbabilitiesAtVertex} follow as Lemmas 3.4, 3.5 and 3.6 in  \cite[Chapter 8]{FreidlinWentzell88}. The main difference between our situation and that of \cite{FreidlinWentzell88} is that we are working with the discrete approximation, which implies that we need information on the error bounds in terms of the parameters $\delta,\veps,\tau$, as it was also the case for Lemma \ref{L:ConvergenceWithinEdge}. Given that the method of the proof is similar to the corresponding proofs of \cite{FreidlinWentzell88}, we do not repeat all the details here.

The principle idea is that Lemma \ref{L:ConvergenceWithinEdge} controls the behavior within each branch of the tree, whereas Lemmas \ref{L:ExteriorVertex}, \ref{L:InteriorVertex} allow us to conclude that the approximating  process spends in the limit zero time on exterior and interior vertices respectively (equivalently it spends zero time in the neighborhood of stable and unstable points of the dynamical system). Then, Lemma \ref{L:ProbabilitiesAtVertex} characterizes the splitting probability in each interior vertex concluding the description of the limiting Markov process.

In order to demonstrate the differences with the corresponding proofs of \cite{FreidlinWentzell88} and to see the role of the discrete approximation,  we  demonstrate the proofs of these lemmas emphasizing the differences.
\begin{proof}[Proof of Lemma \ref{L:ExteriorVertex}]
Let us assume that $U(z_{j})$ is a local minimum of $U$. It is clear that the following relation should hold
\[
|\EE U(\bar{Z}^{\veps}_{\bar{\tau}^{\veps}_{j}(\pm\zeta)})- U(z_{j})|\geq \zeta.
\]

Let us  define
\[
k_{1}=\max\{k\in\mathbb{N}: (k\delta+\tau)\vee k\delta\leq \bar{\tau}^{\veps}_{j}(\pm\zeta)\}.
\]

Let us assume that $k_{1}$ is such that $(k_{1}\delta+\tau)\leq \bar{\tau}^{\veps}_{j}(\pm\zeta)$. The approach is the same if $k_{1}$ is such that $k_{1}\delta\leq \bar{\tau}^{\veps}_{j}(\pm\zeta)$. By adding and subtracting terms of the form $U(\bar{Z}^{\veps}_{m\delta})$ for $m=0,1,\cdots, k_{1}$ we get
\begin{align}
 &U(\bar{Z}^{\veps}_{\bar{\tau}^{\veps}_{j}(\pm\zeta)})= \left[U(\bar{Z}^{\veps}_{\bar{\tau}^{\veps}_{j}(\pm\zeta)})-U(\bar{Z}^{\veps}_{k_{1}\delta+\tau})\right]+U(\bar{Z}^{\veps}_{k_{1}\delta+\tau})\nonumber\\
 &\quad=\left[U(\bar{Z}^{\veps}_{\bar{\tau}^{\veps}_{j}(\pm\zeta)})-U(\bar{Z}^{\veps}_{k_{1}\delta+\tau})\right]+\left[U(\bar{Z}^{\veps}_{k_{1}\delta+\tau})-U(\bar{Z}^{\veps}_{k_{1}\delta})\right]
 +U(\bar{Z}^{\veps}_{k_{1}\delta})\nonumber\\
&\quad=\left[U(\bar{Z}^{\veps}_{\bar{\tau}^{\veps}_{j}(\pm\zeta)})-U(\bar{Z}^{\veps}_{k_{1}\delta+\tau})\right]+\left[U(\bar{Z}^{\veps}_{k_{1}\delta+\tau})-U(\bar{Z}^{\veps}_{k_{1}\delta})\right]
 +\sum_{m=1}^{k_{1}}\left[U(\bar{Z}^{\veps}_{m\delta})-U(\bar{Z}^{\veps}_{(m-1)\delta})\right]+U(z).\nonumber
\end{align}
Taking expected value, Lemma \ref{L:TimeStepIntegrator} (with the test function $f(u) = u$) implies that for $\delta, \tau/\veps$ sufficiently small
\begin{align}
 \EE U(\bar{Z}^{\veps}_{\bar{\tau}^{\veps}_{j}(\pm\zeta)})&=\EE\left[U(\bar{Z}^{\veps}_{\bar{\tau}^{\veps}_{j}(\pm\zeta)})-U(\bar{Z}^{\veps}_{k_{1}\delta+\tau})\right]+\EE\left[U(\bar{Z}^{\veps}_{k_{1}\delta+\tau})-U(\bar{Z}^{\veps}_{k_{1}\delta})\right]\nonumber\\
 &\qquad+\delta \EE \sum_{m=1}^{k_{1}}\left[\mathcal{L}_{0}U(\bar{Z}^{\veps}_{(m-1)\delta})+O\left(\delta^{1/2}+\left(\frac{\tau}{\veps}\right)^{3/2}\frac{1}{\delta}\right)\right]+U(z).\nonumber
\end{align}

Next we notice that up to an unimportant  multiplicative  constant $|\EE U(\bar{Z}^{\veps}_{\bar{\tau}^{\veps}_{j}(\pm\zeta)})-U(z)|<\zeta+\delta$ and that the non-degeneracy of $U$ implies that for every $z\in D_{j}(\pm\zeta)$ there is a constant  $C_{0}>0$, such that $\mathcal{L}_{0}U(z)>C_{0}$. Hence, we have obtained
\begin{align}
 \zeta+\delta&>\EE\left[U(\bar{Z}^{\veps}_{\bar{\tau}^{\veps}_{j}(\pm\zeta)})-U(\bar{Z}^{\veps}_{k_{1}\delta+\tau})\right]+\EE\left[U(\bar{Z}^{\veps}_{k_{1}\delta+\tau})-U(\bar{Z}^{\veps}_{k_{1}\delta})\right]\nonumber\\
 &\quad+ C  \EE \bar{\tau}^{\veps}_{j}(\pm\zeta)\left(1+O\left(\delta^{1/2}+\left(\frac{\tau}{\veps}\right)^{3/2}\frac{1}{\delta}\right)\right) + C  \EE (\delta k_{1}-\bar{\tau}^{\veps}_{j}(\pm\zeta))\left(1+O\left(\delta^{1/2}+\left(\frac{\tau}{\veps}\right)^{3/2}\frac{1}{\delta}\right)\right) \nonumber\\
 &\geq \EE\left[U(\bar{Z}^{\veps}_{\bar{\tau}^{\veps}_{j}(\pm\zeta)})-U(\bar{Z}^{\veps}_{k_{1}\delta+\tau})\right]+\EE\left[U(\bar{Z}^{\veps}_{k_{1}\delta+\tau})-U(\bar{Z}^{\veps}_{k_{1}\delta})\right]\nonumber\\
 &\quad+C_{0}  \EE \bar{\tau}^{\veps}_{j}(\pm\zeta)\left(1+O\left(\delta^{1/2}+\left(\frac{\tau}{\veps}\right)^{3/2}\frac{1}{\delta}\right)\right)
 -\delta \left(1+O\left(\delta^{1/2}+\left(\frac{\tau}{\veps}\right)^{3/2}\frac{1}{\delta}\right)\right). \nonumber
\end{align}

The latter inequality follows since by the definition of $k_1$ we have that $|\delta k_{1}-\bar{\tau}^{\veps}_{j}(\pm\zeta)|<\delta$. Rearranging the latter expression, we obtain for some unimportant constants $0<C_{i}<\infty$
\begin{align}
  \EE \bar{\tau}^{\veps}_{j}(\pm\zeta) &\leq C_{1}\frac{\zeta+2\delta+\delta^{3/2}+(\frac{\tau}{\epsilon})^{3/2}+\left|\EE\left[U(\bar{Z}^{\veps}_{\bar{\tau}^{\veps}_{j}(\pm\zeta)})-U(\bar{Z}^{\veps}_{k_{1}\delta+\tau})\right]\right|+\left|\EE\left[U(\bar{Z}^{\veps}_{k_{1}\delta+\tau})-U(\bar{Z}^{\veps}_{k_{1}\delta})\right]\right|}
 {1+O\left(\delta^{1/2}+\left(\frac{\tau}{\veps}\right)^{3/2}\frac{1}{\delta}\right)} \nonumber\\
&\leq C_{2}\frac{\zeta+2\delta+\delta^{3/2}+(\frac{\tau}{\epsilon})^{3/2}+(\delta+\tau)+(\delta^{3/2}+(\tau/\veps)^{3/2})}
 {1+O\left(\delta^{1/2}+\left(\frac{\tau}{\veps}\right)^{3/2}\frac{1}{\delta}\right)}, \nonumber
\end{align}
which implies that for $0<\tau<\delta<\frac{\tau}{\veps}\ll 1$ and $\left(\frac{\tau}{\veps}\right)^{3/2}\frac{1}{\delta}\downarrow 0$, we get
\begin{align}
 \EE \bar{\tau}^{\veps}_{j}(\pm\zeta) &\leq C_{3}\left(\zeta+\delta+(\frac{\tau}{\epsilon})^{3/2}\right),  \nonumber
\end{align}
or, in other words if we choose $\zeta>\max\{\delta, (\frac{\tau}{\epsilon})^{3/2}\}$, we indeed obtain that
\begin{align}
 \EE \bar{\tau}^{\veps}_{j}(\pm\zeta) &\leq C_{4} \zeta,  \nonumber
\end{align}
from which the statement of the lemma follows.
\end{proof}

\begin{proof}[Proof of Lemma \ref{L:InteriorVertex}]
We start with the following usage of Lemma \ref{L:TimeStepIntegrator},
\begin{align*}
U(\bar{Z}^{\veps}_{n\delta})&=\sum_{k=0}^{n-1}\left[U(\bar{Z}^{\veps}_{(k+1)\delta})-U(\bar{Z}^{\veps}_{k\delta})\right]+ U(z)\nonumber\\
&=\sum_{k=0}^{n-1}\left[U(\bar{Z}^{\veps}_{(k+1)\delta})-U(\bar{Z}^{\veps}_{k\delta+\tau})\right]+ \sum_{k=0}^{n-1}\left[U(\bar{Z}^{\veps}_{k\delta+\tau})-U(\bar{Z}^{\veps}_{k\delta})\right] +U(z)\nonumber\\
&=\sum_{k=0}^{n-1}\left[(\delta-\tau) \mathcal{L}_{0}U(\bar{Z}^{\veps}_{k\delta+\tau})+\tau \mathcal{L}_{0}U(\bar{Z}^{\veps}_{k\delta})\right]+\nonumber\\
&\qquad+\sqrt{2\beta}\sum_{k=0}^{n-1}\left[\sqrt{\delta-\tau} \nabla U(\bar{Z}^{\veps}_{k\delta+\tau})\xi^{'}_{k}+\sqrt{\tau}\nabla U(\bar{Z}^{\veps}_{k\delta})\xi_{k}\right] +\sum_{k=0}^{n-1}\left[R_{2,k}+R_{3,k}\right] +U(z),\nonumber
\end{align*}
where $\xi_{k}^{'},\xi_{k}$ are independent standard normal random variables and $R_{2,k},R_{3,k}$ are as in the proof of Lemma \ref{L:TimeStepIntegrator}. Using the independence of the involved normal random variables, we can then write that in distribution
\begin{align}
U(\bar{Z}^{\veps}_{n\delta})
&=\sum_{k=0}^{n-1}I_{1,k}^{\delta,\tau} + N\left(0,\sum_{k=0}^{n-1}I_{2,k}^{\delta,\tau}\right) +\sum_{k=0}^{n-1}\left[R_{2,k}+R_{3,k}\right] +U(z),
\end{align}
where
\begin{align}
I_{1,k}^{\delta,\tau}&=\left[(\delta-\tau) (\mathcal{L}_{0}U(\bar{Z}^{\veps}_{k\delta+\tau})-\mathcal{L}_{0}U(\bar{Z}^{\veps}_{k\delta}))+\delta \mathcal{L}_{0}U(\bar{Z}^{\veps}_{k\delta})\right],\nonumber\\
I_{2,k}^{\delta,\tau}&=2\beta\left[(\delta-\tau) |\nabla U(\bar{Z}^{\veps}_{k\delta+\tau})|^{2}+\tau|\nabla U(\bar{Z}^{\veps}_{k\delta})|^{2}\right]\nonumber
\end{align}
and $N\left(0,\sum_{k=0}^{n-1}I_{2,k}^{\delta,\tau}\right)$ represents a normal random variable with mean zero and variance $\sum_{k=0}^{n-1}I_{2,k}^{\delta,\tau}$.

Let us recall now that
\[
\left(\EE\left(\sum_{k=0}^{n-1}\left[R_{2,k}+R_{3,k}\right]\right)^{2}\right)^{1/2}=O(n\delta^{3/2}+n(\tau/\veps)^{3/2}).\]

We have that $\bar{\tau}^{\veps}_{j}(\pm\zeta)$ is less or equal to the time when the random variable $|U(\bar{Z}^{\veps}_{n\delta})-U(z)|$ reaches the level $2\zeta$. This happens if the term $\sum_{k=0}^{n-1}\left[I_{1,k}^{\delta,\tau}+R_{2,k}+R_{3,k}\right]$ is small in absolute value, while the term $N\left(0,2\beta\sum_{k=0}^{n-1}(\delta-\tau) |\nabla U(\bar{Z}^{\veps}_{k\delta+\tau})|^{2}+\tau|\nabla U(\bar{Z}^{\veps}_{k\delta})|^{2}\right) $ is large. In other words we have the inclusion
\begin{align}
&\left\{\sum_{k=0}^{n-1}\left[I_{1,k}^{\delta,\tau}+R_{2,k}+R_{3,k}\right]<\zeta, N\left(0,\sum_{k=0}^{n-1}I_{2,k}^{\delta,\tau}\right)>3\zeta\right\} \subseteq \left\{\bar{\tau}^{\veps}_{j}(\pm\zeta)< n\delta\right\}. \label{Eq:LeftInclusion}
\end{align}

We also have
\begin{align}
\left\{\bar{\tau}^{\veps}_{j}(\pm\zeta)\geq n\delta\right\}&\subseteq\left\{\bar{\tau}^{\veps}_{j}(\pm\zeta)\geq n\delta, \sum_{k=0}^{n-1}I_{2,k}^{\delta,\tau}<9\zeta^{2}\right\}\bigcup\nonumber\\
&\quad\bigcup \left\{\bar{\tau}^{\veps}_{j}(\pm\zeta)\geq n\delta, \sum_{k=0}^{n-1}I_{2,k}^{\delta,\tau}\geq 9\zeta^{2}, \left|N(0,9\zeta^{2})\right|\geq 3\zeta\right\}\bigcup\nonumber\\
&\quad\bigcup \left\{\bar{\tau}^{\veps}_{j}(\pm\zeta)\geq n\delta, \sum_{k=0}^{n-1}I_{2,k}^{\delta,\tau}\geq 9\zeta^{2}, \left|N(0,9\zeta^{2})\right|< 3\zeta\right\}.\nonumber
\end{align}

Choose now $n\delta$ such that for the given $\zeta$ we have $n\delta<\zeta$ and in particular that
\[
\sum_{k=0}^{n-1}\left[I_{1,k}^{\delta,\tau}+R_{2,k}+R_{3,k}\right]<\zeta,
\]
for all trajectories $\bar{Z}^{\veps}$ for which
$\bar{\tau}^{\veps}_{j}(\pm\zeta)\geq n\delta$. Then, by
(\ref{Eq:LeftInclusion}), the second inclusion in the last display
cannot hold.  Thus we have
\begin{align}
\PP\left(\bar{\tau}^{\veps}_{j}(\pm\zeta)\geq n\delta\right)&\leq \PP\left(\bar{\tau}^{\veps}_{j}(\pm\zeta)\geq n\delta, \sum_{k=0}^{n-1}I_{2,k}^{\delta,\tau}<9\zeta^{2}\right)+\PP\left( \left|N(0,9\zeta^{2})\right|< 3\zeta\right)\nonumber\\
&=\PP\left(\bar{\tau}^{\veps}_{j}(\pm\zeta)\geq n\delta, \sum_{k=0}^{n-1}I_{2,k}^{\delta,\tau}<9\zeta^{2}\right)+0.6826.\nonumber
\end{align}

Recall that we have chosen $n$ such that $n\delta<\zeta$ and in particular that $\sum_{k=0}^{n-1}\left[I_{1,k}^{\delta,\tau}+R_{2,k}+R_{3,k}\right]<\zeta$. To be precise, the last requirement is that up to a deterministic constant
$n\left(\delta+\delta^{3/2}+(\tau/\veps)^{3/2}\right)<\zeta$. Let us enforce that by requiring that up to an appropriate deterministic constants $\zeta^{2}< n\left(\delta+\delta^{3/2}+(\tau/\veps)^{3/2}\right) <\zeta$. In particular, we can take $n\left(\delta+\delta^{3/2}+(\tau/\veps)^{3/2}\right)$ to be of the order of $\zeta^{2}|\ln \veps|$ such that $\zeta|\ln \veps|\rightarrow 0$. Then, the probability of the first term in the right hand side of the last display can be made as small as we want, say less than $0.10$.

Hence, we have obtained that with the particular choices for $n$ and for sufficiently small $\veps,\delta,\tau$ and $\zeta$ such that $\zeta^{2}< n\left(\delta+\delta^{3/2}+(\tau/\veps)^{3/2}\right) <\zeta$ and $n\left(\delta+\delta^{3/2}+(\tau/\veps)^{3/2}\right)$ to be of the order of $\zeta^{2}|\ln \veps|$, we have that
\begin{align}
\PP\left(\bar{\tau}^{\veps}_{j}(\pm\zeta)\geq n\delta\right)&\leq 0.8.\nonumber
\end{align}
Then, by Markov property we obtain that $\PP\left(\bar{\tau}^{\veps}_{j}(\pm\zeta)\geq N n\delta\right)\leq 0.8^{N}$, which then implies (using the fact that the random variable $\bar{\tau}^{\veps}_{j}(\pm\zeta)$ is positive and that $0.8^{N}$ is a geometric series) that up to a deterministic constant $C<\infty$ that may change from inequality to inequality
\[
\EE_{z}\bar{\tau}^{\veps}_{j}(\pm\zeta)\leq\frac{n\delta}{1-0.8}\leq C \zeta^{2}|\ln \veps|\leq C \zeta^{2}|\ln \zeta|.
\]
  The second to the last inequality of the previous display is true because $n\delta$ is chosen to be of order $\zeta^{2}|\ln \veps|$ and the last inequality because by assumption $\zeta\geq \veps^{\alpha}$ for some exponent $\alpha>0$. This concludes the proof of the lemma.
\end{proof}

\begin{proof}[Proof of Lemma \ref{L:ProbabilitiesAtVertex}]
  Using \cite[Lemma 8.6.2]{FreidlinWentzell88} for the discrete  approximation $\bar{Z}^{\veps}_{t}$ we have
\begin{align}
\lim_{\veps,\delta,\frac{\tau}{\veps}\downarrow 0}\max_{x_{1},x_{2}\in C_{i}(U)}\max_{f:\|f\|\leq 1}\left|\EE_{x_{1}} f (\bar{Z}^{\veps}_{\bar{\tau}^{\veps}(U_{1},U_{2})})-E_{x_{2}} f (\bar{Z}^{\veps}_{\bar{\tau}^{\veps}(U_{1},U_{2})})\right|&=0, \nonumber
\end{align}
where $f$ is defined on $\partial D_{i}(U_{1},U_{2})$ and $\bar{\tau}^{\veps}(U_{1},U_{2})$ is the first time of exit of the process $\bar{Z}^{\veps}$ from  the branch $I_{i}$ from either of the two sides $U_{1}<U_{2}$. Then, by Markov property, as in \cite[Lemma 8.6.3]{FreidlinWentzell88}, we get that
\begin{align}
\lim_{\veps,\delta,\frac{\tau}{\veps}\downarrow 0}\max_{x_{1},x_{2}\in C_{ji}(\zeta^{'})}\left|F^{\veps}(x_{1})-F^{\veps}(x_{2})\right|&=0,\label{Eq:ClosnessOnBoundary}
\end{align}
where for $\zeta^{'}<\zeta$,  $F^{\veps}(x)=\PP\left(\bar{Z}^{\veps}_{\bar{\tau}^{\veps}(\pm\zeta)}\notin \overline{\partial D_{j}(\pm\zeta)}\cap  I^{\circ}_{i}\right)$. The next thing to prove is that for every $\zeta>0, \kappa>0$ there exists $0<\zeta^{'}<\zeta $ such that for every $\veps,\delta,\frac{\tau}{\veps}$ sufficiently small
\begin{align}
\max_{x_{1},x_{2}\in \bar{D_{j}}(\pm \zeta^{'})}\left|F^{\veps}(x_{1})-F^{\veps}(x_{2})\right|&<\kappa.\nonumber
\end{align}

For edges $I_{i}\sim O_{j}$ let us set $f(x)=1_{x \notin \overline{\partial D_{j}(\pm\zeta)}\cap  I^{\circ}_{i}}$. There are exactly three regions corresponding to $I^{\circ}_{i_{0}}, I^{\circ}_{i_{1}}, I^{\circ}_{i_{2}}$ that are separated by the separatrix $C_{j}$. The region corresponding to $I^{\circ}_{i_{0}}$ adjoins the whole curve $C_{j}$, whereas $I^{\circ}_{i_{1}}, I^{\circ}_{i_{2}}$ adjoins only part of it. In particular we have that $C_{ji_{0}}=C_{ji_{1}}\cup C_{ji_{2}}$. Then, as in the proof of \cite[Lemma 8.3.6]{FreidlinWentzell88}, it can be shown that
\begin{equation}
\begin{aligned}
\left|F^{\veps}(x_{1})-F^{\veps}(x_{2})\right|&\leq \left[\sup_{x\in\partial D_{j}(\pm\zeta)}f(x)-\inf_{x\in\partial D_{j}(\pm\zeta)}f(x)\right]\\
&\qquad \times \max\left\{ \PP_{x_{m}}\left(\bar{Z}^{\veps}_{\bar{\tau}}\notin (\overline{\partial D_{j}(\pm\zeta)}\cap  I^{\circ}_{i_{1}})\cup (\overline{\partial D_{j}(\pm\zeta)}\cap  I^{\circ}_{i_{2}})\right): m=1,2\right\}\\
&+ \left[\sup_{x\in C_{ji_{0}}(\zeta')}F^{\veps}(x)-\inf_{x\in C_{ji_{0}}(\zeta')}F^{\veps}(x)\right],\label{Eq:ExpressionToBound}
\end{aligned}
\end{equation}
where $\bar{\tau}$ is the first time that the discrete approximation process $\bar{Z}^{\veps}_{t}$ exits $(\overline{\partial D_{j}(\pm\zeta')}\cap  I^{\circ}_{i_{0}})$ or $(\overline{\partial D_{j}(\pm\zeta)}\cap  I^{\circ}_{i_{1}})\cup (\overline{\partial D_{j}(\pm\zeta)}\cap  I^{\circ}_{i_{2}})$. Clearly, we have that $\bar{\tau}\leq \bar{\tau}^{\veps}(\pm\zeta)$.

By (\ref{Eq:ClosnessOnBoundary}), the second additive term in (\ref{Eq:ExpressionToBound}) is arbitrarily small for sufficiently small $\veps,\delta,\tau/\veps$. So, it remains to estimate $P^{\veps}_{x}\doteq \PP_{x}\left(\bar{Z}^{\veps}_{\bar{\tau}}\notin (\overline{\partial D_{j}(\pm\zeta)}\cap  I^{\circ}_{i_{1}})\cup (\overline{\partial D_{j}(\pm\zeta)}\cap  I^{\circ}_{i_{2}})\right)$. As in the proof of Lemma \ref{L:ExteriorVertex} for an appropriate integer $k_{1}$ and for $\delta, \tau/\veps$ sufficiently small
\begin{align}
 \EE U(\bar{Z}^{\veps}_{\bar{\tau}})&=\EE\left[U(\bar{Z}^{\veps}_{\bar{\tau}})-U(\bar{Z}^{\veps}_{k_{1}\delta+\tau})\right]+\EE\left[U(\bar{Z}^{\veps}_{k_{1}\delta+\tau})-U(\bar{Z}^{\veps}_{k_{1}\delta})\right]\nonumber\\
 &\qquad+\delta \EE \sum_{m=1}^{k_{1}}\left[\mathcal{L}_{0}U(\bar{Z}^{\veps}_{(m-1)\delta})+O\left(\delta^{1/2}+\left(\frac{\tau}{\veps}\right)^{3/2}\frac{1}{\delta}\right)\right]+U(z).\label{Eq:EstimatingExitProb}
\end{align}

Notice now that $U(\bar{Z}^{\veps}_{\bar{\tau}})$ is either greater or equal than $U(O_{j})\pm \zeta'$ on $(\overline{\partial D_{j}(\pm\zeta')}\cap  I^{\circ}_{i_{0}})$ or it is greater or equal than $U(O_{j})\mp \zeta$ on
$(\overline{\partial D_{j}(\pm\zeta)}\cap  I^{\circ}_{i_{1}})\cup (\overline{\partial D_{j}(\pm\zeta)}\cap  I^{\circ}_{i_{2}})$. The latter implies that
\begin{align}
 \EE U(\bar{Z}^{\veps}_{\bar{\tau}})&\geq (U(O_{j})\pm \zeta')(1-P^{\veps}_{x})+ (U(O_{j})\mp \zeta)P^{\veps}_{x}. \nonumber
 \end{align}

The latter and (\ref{Eq:EstimatingExitProb}) imply that up to deterministic constants that do not depend on the small parameters of the problem
\begin{align}
 (\zeta+\zeta^{'})P^{\veps}_{x}&\leq \zeta' + |U(O_{j})-U(z)| +|\EE\left[U(\bar{Z}^{\veps}_{\bar{\tau}^{\veps}_{j}(\pm\zeta)})-U(\bar{Z}^{\veps}_{k_{1}\delta+\tau})\right]|+|\EE\left[U(\bar{Z}^{\veps}_{k_{1}\delta+\tau})-U(\bar{Z}^{\veps}_{k_{1}\delta})\right]|\nonumber\\
 &\qquad+\left|\delta \EE \sum_{m=1}^{k_{1}}\left[\mathcal{L}_{0}U(\bar{Z}^{\veps}_{(m-1)\delta})+O\left(\delta^{1/2}+\left(\frac{\tau}{\veps}\right)^{3/2}\frac{1}{\delta}\right)\right]\right|\nonumber\\
 &\leq 2\zeta' +|\EE\left[U(\bar{Z}^{\veps}_{\bar{\tau}^{\veps}_{j}(\pm\zeta)})-U(\bar{Z}^{\veps}_{k_{1}\delta+\tau})\right]|+|\EE\left[U(\bar{Z}^{\veps}_{k_{1}\delta+\tau})-U(\bar{Z}^{\veps}_{k_{1}\delta})\right]|\nonumber\\
 &\qquad+  \EE \bar{\tau}\left(1+O\left(\delta^{1/2}+\left(\frac{\tau}{\veps}\right)^{3/2}\frac{1}{\delta}\right)\right) +   \left|\EE (\delta k_{1}-\bar{\tau})\right|\left(1+O\left(\delta^{1/2}+\left(\frac{\tau}{\veps}\right)^{3/2}\frac{1}{\delta}\right)\right) \nonumber\\
 &\leq 2\zeta' +\delta+\tau+ \delta^{3/2}+\left(\frac{\tau}{\veps}\right)^{3/2}+  \EE \bar{\tau}\left(1+O\left(\delta^{1/2}+\left(\frac{\tau}{\veps}\right)^{3/2}\frac{1}{\delta}\right)\right)  \nonumber\\
 &\leq 2\zeta' +\delta+\tau+ \delta^{3/2}+\left(\frac{\tau}{\veps}\right)^{3/2}+  \EE \bar{\tau}^{\veps}(\pm\zeta)\left(1+O\left(\delta^{1/2}+\left(\frac{\tau}{\veps}\right)^{3/2}\frac{1}{\delta}\right)\right)  \nonumber\\
 &\leq 2\zeta' +\delta+\tau+ \delta^{3/2}+\left(\frac{\tau}{\veps}\right)^{3/2}+  \zeta^{2}|\ln \zeta|\left(1+O\left(\delta^{1/2}+\left(\frac{\tau}{\veps}\right)^{3/2}\frac{1}{\delta}\right)\right).  \nonumber
\end{align}
where for the last line we used Lemma \ref{L:InteriorVertex}. Therefore, we have obtained that for sufficiently small $\tau<\delta\ll 1$ such that $\tau/\veps\downarrow 0$
\begin{align}
\PP_{x}\left(\bar{Z}^{\veps}_{\bar{\tau}}\notin (\overline{\partial D_{j}(\pm\zeta)}\cap  I^{\circ}_{i_{1}})\cup (\overline{\partial D_{j}(\pm\zeta)}\cap  I^{\circ}_{i_{2}})\right)
&\leq \frac{2\zeta'}{\zeta}+\zeta|\ln \zeta|+\frac{\delta+\tau+ \delta^{3/2}+\left(\frac{\tau}{\veps}\right)^{3/2}}{\zeta+\zeta'}\nonumber\\
&\leq \frac{2\zeta'}{\zeta}+\zeta|\ln \zeta|+\frac{\delta+\left(\frac{\tau}{\veps}\right)^{3/2}}{\zeta+\zeta'}.\nonumber
\end{align}

The right hand side of the last display can be made arbitrarily small, if we choose $\zeta'<\zeta$ small but such that $\frac{\delta+\left(\frac{\tau}{\veps}\right)^{3/2}}{\zeta+\zeta'}\downarrow 0$. This means that $\zeta^{'}<\zeta$ should be chosen small, but greater than $\delta+\left(\frac{\tau}{\veps}\right)^{3/2}$. Hence, under this condition, we get that $\PP_{x}\left(\bar{Z}^{\veps}_{\bar{\tau}}\notin (\overline{\partial D_{j}(\pm\zeta)}\cap  I^{\circ}_{i})\right)$ has approximately the same value for all $x\in \bar{D}_{j}(\pm\zeta')$ when $\tau<\delta<\frac{\tau}{\veps}\ll 1$ and $\left(\frac{\tau}{\veps}\right)^{3/2}\frac{1}{\delta}\ll 1$. Then, it remains to show that, in the limit, this value is actually equal to $p_{ji}$. This part of the proof however follows very closely the corresponding part of the proof of \cite[Lemma 3.6]{FreidlinWentzell88} for $Z^{\veps}$ when $\delta,\tau/\veps$ are sufficiently small and it will not be repeated here. This concludes  the proof of the lemma.
\end{proof}

\bibliographystyle{amsxport}

\end{document}